\documentclass[11pt]{amsart}

\usepackage{amsmath, amssymb}
\usepackage{amsfonts}
\usepackage{graphicx}
\usepackage{mathrsfs}
\usepackage[arrow,matrix,curve,cmtip,ps]{xy}
\usepackage{amsthm}
\usepackage{enumerate}
\usepackage{color}
\usepackage[colorlinks]{hyperref}
\usepackage{tikz}

\allowdisplaybreaks

\newtheorem{thm}{Theorem}[section]
\newtheorem{lem}[thm]{Lemma}
\newtheorem{prop}[thm]{Proposition}
\newtheorem{cor}[thm]{Corollary}
\newtheorem*{theorem*}{Theorem}
\theoremstyle{remark}
\newtheorem{rem}[thm]{Remark}

\newtheorem{defn}[thm]{Definition}
\newtheorem{example}[thm]{Example}

\numberwithin{equation}{section}

\newcommand{\om}{\omega}

\begin{document}
\title[Self-similar ultragraph $C^*$-algebras]{Self-similar group actions on ultragraphs and associated $C^*$-algebras}

\author[H.Larki and N. Rajabzadeh-Hasiri]{Hossein Larki and Najmeh Rajabzadeh-Hasiri}

\address{Department of Mathematics\\
Faculty of Mathematical Sciences and Computer\\
Shahid Chamran University of Ahvaz\\
P.O. Box: 83151-61357\\
Ahvaz\\
 Iran}
\email{h.larki@scu.ac.ir, rajabzadeh.najmeh1396@gmail.com}


\date{\today}



\begin{abstract}
As a generalization of the Exel-Pardo's notion of self-similar graph, we introduce self-similar group actions on ultragraphs and their $C^*$-algebras. We then approach to the $C^*$-algebras by inverse semigroup and tight groupoid models.
\end{abstract}

\maketitle


\section{Introduction}

Ultragraph $C^*$-algebras were originally introduced by Tomforde \cite{ana00} to give a unified framework for the graph $C^*$-algebras \cite{nek04,li19} and Exel-Laca algebras \cite{li18}. An ultragraph is in fact a generalized (directed) graph in which the source of each edge is a nonempty set of vertices rather than a single vertex. So, we expect that many ideas and results for the graph $C^*$-algebras in the literature could be extended (with more effort) to the ultragraph setting \cite{pin13,nek05,ana00,anal09}. Furthermore, the algebraic analogue of the ultragraph $C^*$-algebras, namely ultragraph Leavitt path algebras, has attracted attension in recent years (see \cite{dlpo890,bed17,li1345tuih,li1p0987} among others).

The notion of self-similar graph $C^*$-algebras was introduced by Exel and Pardo \cite{bro14} to bring two important classes of $C^*$-algebras, the Katsura's \cite{li1po09} and Nekrashevych's algebras \cite{exe18} under one theory like graph $C^*$-algebras. In particular, it contains all UCT Kirchberg $C^*$-algebras \cite{bro14} and gives examples of important groupoid $C^*$-algebras with non-Hausdorff underlying groupoid \cite{li12}.

The aim of this article is to generalize the Exel-Pardo $C^*$-algebras \cite{bro14} to self-similar ultragraphs, and then give an inverse semigroup model for them. Note that self-similar graphs $(G,E,\varphi)$ in \cite{bro14} are assumed to satisfy
$$\varphi(g,e)\cdot s(e)=g \cdot s(e)$$
for all $g\in G$ and $e \in E^1$. This condition is necessary to define a cyclic action of the group $G$ on the path space $E^*$ of $E$ by
$$g \cdot (\alpha\beta)=(g\cdot \alpha)(\varphi(g,\alpha)\cdot \beta)$$
where $\alpha,\beta \in E^*$ with $s(\alpha)=r(\beta)$. However, for the self-similar ultragraphs we just need
\begin{equation}\label{eq1}
\varphi(g,e)\cdot s(e)\subseteq g\cdot s(e),
\end{equation}
because the source of each edge is a set of vertices (see Remark \ref{rem3.3} below). In particular, by this weaker condition, we can easily construct finite self-similar ultragraphs which are not presented by self-similar graphs of \cite{bro14}. Furthermore, we believe that our construction of self-similar ultragraph $C^*$-algebras can be also modified and extended to the labelled graph setting. (A labelled graph is another generalization of directed graphs in which both source and range of each edge are set of vertices.)

This article is organized as follows. After reviewing necessary background in Section \ref{sec2}, we introduce a self-similar ultragraph system $(G,\mathcal{U},\varphi)$ in Section \ref{sec30}, which is a self-similar group action $G  \curvearrowright   \mathcal{U}$ on an ultragraph $\mathcal{U}=(U^0,\mathcal{U}^1,r,s)$ equipped with a $1$-cocycle $G \times \mathcal{U}^1 \longrightarrow G$ satisfying condition (\ref{eq1}) above, and we then associate to $(G,\mathcal{U},\varphi)$ a universal $C^*$-algebra ${\mathcal{O}}_{G,\mathcal{U}}$ satisfying certain conditions. In Section \ref{sec4}, we consider some clases of examples. First, we show that for source-finite self-similar ultragraphs with a condition, the $C^*$-algebras ${\mathcal{O}}_{G,\mathcal{U}}$ are included in the Exel-Pardo $C^*$-algebras of self-similar graphs. However, general $C^*$-algebras ${\mathcal{O}}_{G,\mathcal{U}}$ are not; we construct a simple finite self-similar ultragraph which is not of the form self-similar graphs in \cite{bro14}. Next, we define a self-similar analogue of Exel-Laca algebras \cite{li18} as another class of examples.

 In Section \ref{sec5}, for each self-similar ultragraph $(G,\mathcal{U},\varphi)$ we define an inverse semigroup ${\mathcal{S}}_{G,\mathcal{U}}$, and the rest of article is devoted to explain and analyze this inverse semigroup. Since ultragraphs contain generalized sets of vertices instead of singletons, elements of ${\mathcal{S}}_{G,\mathcal{U}}$ are of the form $(\alpha,A,g,\beta)$, while we have triples $(\alpha,g,\beta)$ in ${\mathcal{S}}_{G,E}$ of \cite{bro14}. It turns out in particular that ${\mathcal{S}}_{G,\mathcal{U}}$ has a more complicated idempotent semilattice $\mathcal{E}({\mathcal{S}}_{G,\mathcal{U}})$, among others. So, after recalling the description of the tight filter space ${\widehat{\mathcal{E}}}_{\mathrm{tight}}({\mathcal{S}}_{G,\mathcal{U}})$ from \cite{bed17} and the construction of the tight groupoid ${\mathcal{G }_{\mathrm{tight}}}({\mathcal{S}_{G,\mathcal{U}}})$ in Section \ref{sec80}, we compute in Section \ref{sec800} the action ${\mathcal{S}}_{G,\mathcal{U}}     \curvearrowright     {\widehat{\mathcal{E}}}_{\mathrm{tight}}({\mathcal{S}}_{G,\mathcal{U}})$ carefully, which may be of interest even for ordinary ultragraphs $\mathcal{U}$. Then in Section \ref{sec8}, we prove that the canonical representation $\pi : {\mathcal{S}}_{G,\mathcal{U}}\longrightarrow {\mathcal{O}}_{G,\mathcal{U}}$ is a universal tight one, and therefore $ {\mathcal{O}}_{G,\mathcal{U}}\cong C^*_{\mathrm{tight}}({\mathcal{S}}_{G,\mathcal{U}})\cong  C^*({\mathcal{G }_{\mathrm{tight}}}({\mathcal{S}_{G,\mathcal{U}}}))$.

Finally, in the last two sections, we characterize the Hausdorffness of the groupoid ${\mathcal{G }_{\mathrm{tight}}}({\mathcal{S}_{G,\mathcal{U}}})$ and the $E^*$-unitary property of ${\mathcal{S}}_{G,\mathcal{U}}$ which generalize know results of \cite{bro14}. We should notice that structural properties such as minimality and effectiveness of ${\mathcal{G }_{\mathrm{tight}}}({\mathcal{S}_{G,\mathcal{U}}})$ are more complicated than those in the self-similar graph setting, and we leave them to future works.


\section{Preliminaries}\label{sec2}

In this section we recall the necessary background for the rest of the paper.

\subsection{Ultragraphs}
An $\textit{ultragraph}$  is a quadruple $\mathcal{U}=(U^0,\mathcal{U}^1,r,s)$, where $U^0$ denoted the set of vertices, $\mathcal{U}^1$ is the set of edges and $r,s:\mathcal{U}^1\rightarrow P({U}^0)\setminus\{\varnothing\}$, describing the range and source maps, where $P({U}^0)$  is the power set of ${U}^0$ and $r(e)$ is a singleton for every $e\in{\mathcal{U}^1}$. A {\it{source}} in $\mathcal{U}$ is a vertex $v\in{U}^0$ satisfying $\ r^{-1}(\{v\})=\varnothing$. Also, we say $\mathcal{U}$ is $\textit{row-finite}$ whenever each vertex receives at most finitely many edges, that is $\ r^{-1}(\{v\}) $ is a finite set for every $v \in{U^0}$.

\begin{rem}\label{rem2.2}
We prefer to reverse the edge's direction of ultragraphs in the literature \cite{bed17,pin13,ana00}, which will be compatible with self-similar graphs and $k$-graphs in \cite{bro14,bur81,opli9,cla15,cla18,lar2023}.
\end{rem}

{\bf Standing assumption.} Throughout the paper, we will assume that all ultragraphs $\mathcal{U}$ are row-finite and without any source, meaning that the set $r^{-1}(\{v\})=\{e\in \mathcal{U}^1 \;:\; r(e)=\{v\}\}$ is nonempty and finite for all $v\in U^0$.\\

Let $\mathcal{U}$ be an ultragraph. We denote by ${\mathcal{U}}^0$ the smallest subset of $P({U}^0)$ containing $\varnothing$, $\{v\}$ for all $v\in{U}^0$,  $s(e)$ for all $e\in{\mathcal{U}}^1$, and being closed under the  finite set operations $\cap$, $\cup$ and $\backslash$. A  \textit{finite path} in ${\mathcal{U}}$ is either an element of ${\mathcal{U}}^0$ or a sequence of edges $\alpha= e_1\cdots  e_n$ in ${\mathcal{U}}^1$ such that $r(e_{i+1})\subseteq  s(e_i)$ for $1\leq i\leq n$. Then the {\it{range}} of  $\alpha$ is defined by $r(\alpha)=r(e_1)$, the {\it{source}} of $\alpha$ by $s(\alpha)=s(e_n)$, and the {\it{length}} of $\alpha$ by $|\alpha|=n$. Each element ${A}\in{\mathcal{U}}^0$ will be usually considered as a path of length zero in which case we set $r(A)=s(A)=A$.  The set of all finite paths in $\mathcal{U}$ is denoted by ${\mathcal{U}}^*=\cup_{n=o}^{\infty}{\mathcal{U}}^n$. An \textit{infinite path} in  $\mathcal{U}$ is an infinite sequence of edges $\alpha=e_1e_{2} \cdots$, where $r(e_{i+1})\subseteq  s(e_i)$ for all $i\geq1$. The set of infinite paths in $\mathcal{U}$ is denoted by ${\mathcal{U}}^\infty$. If $\alpha$ and $\beta$ are two paths in $\mathcal{U}^{\geq 1}$ such that $r(\beta)\subseteq  s(\alpha)$, we will denote by $\alpha\beta$ the path obtained by juxtaposing $\alpha$ and $\beta$.

\begin{rem}
In \cite{ana00}, Tomforde defined ${\mathcal{U}}^0$ to be the smallest subset of $P({\mathcal{U}}^0)$ containing $\{\{v\},\; s(e)\;:\; v\in U^0 \; \text{and} \; e \in  {\mathcal{U}}^1\}$, which is only closed under the operation $\cap$ and $\cup$. As noted in \cite[Remark 2.5]{li19-boundary}, the further assumption that ${\mathcal{U}}^0$ is closed under the set minus operation dose not change the generated $C^*$-algebra by a Cuntz-Krieger ${\mathcal{U}}$-family. However, it sometimes makes our arguments simpler; in particular in the proofs in Section \ref{sec8}.
\end{rem}

\begin{defn}[\cite{ana00}]\label{defn2.3}
Let $\mathcal{U}$ be a row-finite source-free ultragraph. A {\it Cuntz-Krieger $\mathcal{U}$-family} is a collection of partial isometries $\{s_e:e\in {\mathcal{U}}^{1}\}$ with mutually orthogonal ranges and collection of projections
 $\{p_A:A\in{\mathcal{U}}^{0}\}$ satisfying
\begin{enumerate}[(CK1)]
  \item $p_{\emptyset}=0,\;p_{A}p_{B}=p_{A\cap{B}}$ and $p_{A\cup{B}}=p_{A}+p_{B}-p_{A\cap{B}}$ for all $A,B\in {\mathcal{U}}^{0}$,
  \item $s_e^*s_e=p_{s(e)}$ for all $e\in {\mathcal{U}}^1$,
  \item $s_es_e^*\leq p_{r(e)}$ for all $e\in {\mathcal{U}}^1$, and
  \item $p_{\{v\}}=\sum_{r(e)=\{v\}}s_e s_e^*$ for all $v\in {U}^0$.
\end{enumerate}
\end{defn}

Recall from \cite[Lemma 2.8]{ana00} that if $A\in{\mathcal{U}}^0$  and $e\in{\mathcal{U}}^1$, then
$$p_As_e=\left\{
           \begin{array}{ll}
             s_e & \text{if } r(e)\subseteq  A \\
             0 & \text{otherwise}
           \end{array}
         \right.
\text{ and }
s^*_ep_A=\left\{
  \begin{array}{ll}
    s^*_e & \text{if}\;r(e)\subseteq  A\\
    0 & \text{otherwise}.
  \end{array}
\right.
$$
Moreover, we write $s_\alpha:=s_{e_1} \ldots s_{e_n}$ if $\alpha=e_1 \ldots e_n\in{\mathcal{U}}^n$, and $s_\alpha:=p_A$
if $\alpha=A\in{\mathcal{U}}^0$; so $p_A s_{\alpha}=p_A s_{e_1} \ldots s_{e_n}=s_{\alpha}$ whenever $\alpha=e_1 \ldots e_n$ with $r(\alpha)\subseteq A$, while it is zero otherwise.

\subsection{Groupoid $C^*$-algebras}

Throughout the paper, we consider a specific type of topological groupoids, called \emph{tight groupoids of germs}, and here we recall a brief introduction of topological groupoids and their $C^*$-algebras. For more details, see \cite{don14,exe08}.

A {\it groupoid} is a small category $\mathcal{G}$ with an inverse map $\lambda\mapsto \lambda^{-1}$,  the range map $r(\lambda)=\lambda \lambda^{-1}$ and also the source map $s(\lambda)=\lambda^{-1}\lambda$ for all $\lambda\in \mathcal{G}$, satisfying $r(\lambda)\lambda=\lambda=\lambda s(\lambda)$. Therefore, for every morphisms $\lambda,\eta\in \mathcal{G}$, the composition $\lambda \eta$ is well-defined if and only if $s(\lambda)=r(\eta)$. Moreover, the set of identity morphisms is called {\it the unit space of $\mathcal{G}$}, that is $\mathcal{G}^{(0)}:=\{\lambda^{-1}\lambda:\lambda\in \mathcal{G}\}$.

We will only deal with {\it {topological groupoids}} $\mathcal{G}$ in which multiplication, inverse, range and source maps are continuous. In this case, $\mathcal{G}$ is called an {\it \'{e}tale groupoid} if its topology is locally compact and $r,s$ are local homeomorphisms. A {\it bisection} is a subset $B\subseteq \mathcal{G}$ such that the restricted maps $r|_B$ and $s|_B$ are both homeomorphisms. Also $\mathcal{G}$ is {\it ample} in case the topology on $\mathcal{G}$ is generated by a basis of compact open bisections.

Since all tight groupoids of inverse semigroups in this paper are ample, we recall the full and reduced $C^*$-algebras associated to  ample groupoids. Fixing an ample groupoid $\mathcal{G}$, write $C_c(\mathcal{G})$ for the $*$-algebra of compactly supported continuous functions $f:\mathcal{G} \rightarrow \mathbb{C}$ with the convolution multiplication and the involution $f^*(\lambda):=\overline{f(\lambda^{-1})}$. For any unit $u\in \mathcal{G}^{(0)}$, define the left regular $*$-representation $\pi_u:C_c(\mathcal{G})\rightarrow B(\ell^2(\mathcal{G}_u))$ with
$$\pi_u(f)\delta_\lambda:=\sum_{s(\eta)=r(\lambda)}f(\eta) \delta_{\eta \lambda} \hspace{5mm} (f\in C_c(\mathcal{G}),~\lambda\in \mathcal{G}_u),$$
where $\mathcal{G}_u:=s^{-1}(\{u\})$. Then the {\it reduced $C^*$-algebra $C^*_r(\mathcal{G})$} is the completion of $C_c(\mathcal{G})$ by reduced $C^*$-norm
$$\|f\|_r:=\sup_{u\in \mathcal{G}^{(0)}}\|\pi_u(f)\|.$$
There is also a {\it full $C^*$-algebra $C^*(\mathcal{G})$} associated to $\mathcal{G}$, which is the completion of $C_c(\mathcal{G})$ taken over all $\|.\|_{C_c(\mathcal{G})}$-decreasing representations of $\mathcal{G}$. Hence, $C^*_r(\mathcal{G})$ is a quotient of $C^*(\mathcal{G})$, and by \cite[Proposition 6.1.8]{don14} they are equal whenever $\mathcal{G}$ is amenable.

\subsection{Inverse semigroups and their tight groupoids}\label{sub2.3}

An \emph{inverse semigroup} is a semigroup $S$ which for each $s\in S$, there is a unique $s^*\in S$ satisfying
$$s=s s^* s \hspace{5mm} \mathrm{and} \hspace{5mm} s^*=s^*s s^*.$$
In particular, we have $e^*=e$ for all idempotents in $S$ because $e=eee$. We denote the set of all idempotents in $S$ by $\mathcal{E}(S)$, which is a commutative semilattice by defining $e \wedge f:=ef$. For $e,f\in \mathcal{E}(S)$, we say that $e$ intersects $f$, denoted by $e\Cap f$, if $ef\neq 0$. There is a natural partial order on $S$ by $s\leq t$ if $s=te$ for some $e\in \mathcal{E}(S)$. According to this partial order, for any $s\in S$, we will use the notation
$$s\uparrow:=\{t\in S:t\geq s\}.$$

\begin{defn}
Let $S$ be an inverse semigroup containing a zero element $0$ (in the sense $0s=s0=0$ for every $s\in S$). A {\it filter} in $\mathcal{E}(S)$ is a nonempty set $\mathcal{F}\subseteq \mathcal{E}(S)$ such that it is closed under multiplication and that if $s \in \mathcal{F} $ then $ s\uparrow \subseteq \mathcal{F}$. An {\it ultrafilter} is a proper maximal filter in $\mathcal{E}(S)$. In addition the set of all filters (ultrafilters) without 0 in $\mathcal{E}(S)$ is denoted by $\widehat{\mathcal{E}}_0(S)$ ($\widehat{\mathcal{E}}_\infty(S)$ respectively).
\end{defn}

Note that we can equip $\widehat{\mathcal{E}}_0(S)$ with a locally compact and Hausdorff topology generated by neighborhoods of the form
$$N(e;e_1,\ldots ,e_n):=\{\mathcal{F}\in \widehat{\mathcal{E}}_0(S): e\in \mathcal{F}, e_1,\ldots,e_n\notin \mathcal{F}\} \hspace{5mm} (e,e_1,\ldots,e_n\in \mathcal{E}(S)).$$

\begin{defn}[\cite{exe21}]
The {\it  tight filter space } (or {\it tight spectrum}) of $S$ denoted by $\widehat{\mathcal{E}}_{\mathrm{tight}}(S)$, is the closure of $\widehat{\mathcal{E}}_\infty(S)$ in $\widehat{\mathcal{E}_0}(S)$ and each filter in $\widehat{\mathcal{E}}_{\mathrm{tight}}(S)$ is called a {\it tight filter}.
We may consider $\widehat{\mathcal{E}}_{\mathrm{tight}}(S)$ as a topological space equipped with the restricted topology from $\widehat{\mathcal{E}_0}(S)$.
\end{defn}

Following \cite{exe21}, there is a natural action of $S$ on $\widehat{\mathcal{E}}_{\mathrm{tight}}(S)$, and hence one may construct its groupiod of germs. Briefly, for any $e\in \mathcal{E}(S)$, let
$$D^e:=\{\mathcal{F}\in \widehat{\mathcal{E}}_{\mathrm{tight}}(S): e\in \mathcal{F}\}.$$
Then the action $\theta: S\curvearrowright \widehat{\mathcal{E}}_{\mathrm{tight}}(S)$, $s\mapsto \theta_s$, is defined in terms of the maps $\theta_s:D^{s^* s}\rightarrow D^{s s^*}$ by $\theta_s(\mathcal{F})=(s\mathcal{F}s^*)\uparrow$. Consider the set
$$S*\widehat{\mathcal{E}}_{\mathrm{tight}}(S)=\{(s,\mathcal{F})\in S\times \widehat{\mathcal{E}}_{\mathrm{tight}}(S): s^*s\in \mathcal{F}\}$$
and define $\thicksim$ on $S*\widehat{\mathcal{E}}_{\mathrm{tight}}(S)$ by $(s,\mathcal{F})\sim (t,\mathcal{F}')$ whenever $\mathcal{F}=\mathcal{F}'$ and $se=te$
for some $e\in \mathcal{F}$. For each $(s,\mathcal{F})$, write $[s,\mathcal{F}]$ for the equivalent class of $(s,\mathcal{F})$. Then the {\it tight groupoid associated to $S$} is the groupoid $\mathcal{G}_{\mathrm{tight}}(S)=S*\widehat{\mathcal{E}}_{\mathrm{tight}}(S)/\sim$ with the multiplication
$$[t,\theta_s(\mathcal{F})][s,\mathcal{F}]:=[ts,\mathcal{F}]$$
and the inverse map
$$[s,\mathcal{F}]^{-1}:=[s^*,\theta_s(\mathcal{F})].$$
In addition, the source and  range maps are
$$ s([s,\mathcal{F}])=[s^*s,\mathcal{F}]  \hspace{5mm} \mathrm{and} \hspace{5mm}  r([s,\mathcal{F}])=[ss^*,\theta_s(\mathcal{F})] $$
and we may identify the unit space of $\mathcal{G}_{\mathrm{tight}}(S)$ with $\widehat{\mathcal{E}}_{\mathrm{tight}}(S)$ that $[e,\mathcal{F}]\mapsto \mathcal{F}$.\\


\section{Self-similar ultragraphs and their  $C^*$-algebras}\label{sec30}

In this section, we introduce self-similar ultragraphs and associated $C^*$-algebras. They contain naturally the class of self-similar graph $C^*$-algebras \cite{bro14}, and we will see in the next section the containment is proper.

 Let $\mathcal{U}=(U^0,\mathcal{U}^1,r,s)$ be an ultragraph. By an \textit{automorphism} on $\mathcal{U}$, we mean a bijective map
$$\sigma:{U}^0\sqcup{\mathcal{U}}^1\rightarrow{U}^0\sqcup{\mathcal{U}}^1$$
 such that $\sigma(U^0)\subseteq{U^0}$,
$\sigma({\mathcal{U}}^1)\subseteq{{\mathcal{U}}^1}$ and moreover that $r(\sigma(e))=\sigma(r(e))$,  $s(\sigma(e))=\sigma(s(e))$ for all $e\in{\mathcal{U}}^1 $. Then the collection of all automorphisms of $\mathcal{U}$ forms a group under composition. Let $G$ be a contable discrete group. An {\it{action}} $G\curvearrowright \mathcal{U}$
 is a map $G \times (U^0\sqcup {\mathcal{U}}^1)\rightarrow U^0\sqcup {\mathcal{U}}^1$, denoted by $(g,a)\mapsto g\cdot a$, such that the action of each $g\in G$ on $\mathcal{U}$ gives an ultragraph automorphism. In this case, we define
 $${g \cdot A}:=\left\{g \cdot v\;\;:{v}\in{A}\right\}$$
for ${g}\in{G}$ and $A\in {\mathcal{U}}^o.$

\begin{lem}\label{lem3.1}
Let $G\curvearrowright \mathcal{U}$ be an action of $G$ on $ \mathcal{U}$ as above. For every $g\in{G}$  and $A \in {\mathcal{U}^0}$, we have $g \cdot A \in {\mathcal{U}^0}$ as well.
\end{lem}
\begin{proof}
The set operations $\cap,\cup$, and $\backslash$ are preserved by the maps $A \longmapsto g \cdot A$; for example, if $A=s(e)\cap s(f)$ then
\begin{align}
g\cdot A&=g\cdot ( s(e)\cap s(f)  )\nonumber\\
&=( g\cdot  s(e)) \cap  ( g\cdot s(f) )\nonumber\\
&=s(g\cdot e)\cap s(g\cdot f),\nonumber
\end{align}
or similarly, if $A=s(e) \backslash  s(f) $ then
$$g\cdot A=g \cdot (  s(e)\backslash  s(f)  )=s(g\cdot e)\backslash s(g\cdot f).$$
Moreover, by definition, every $A \in {\mathcal{U}}^0$ is constructed from $\{ v,s(e)\; :\; v\in U^0, e\in   {\mathcal{U}}^1 \}$ via finite operations $\cap,\cup$, and $\backslash$. Combining these observations proves the statement.
\end{proof}

We now define a self-similar ultragraph.

\begin{defn}\label{defn3.2}
Let $\mathcal{U}$ be a row-finite ultragraph without sources and $G$ a discrete group with identity ${{1}_G}$. A \textit{self-similar ultragraph} is a triple $(G,\mathcal{U},\varphi)$ such that

\begin{enumerate}
  \item $\mathcal{U}$ is an ultragraph,
  \item $G$ acts on $\mathcal{U}$ by automorphisms, and
  \item $\varphi:G\times \mathcal{U}^1\rightarrow G$ is a 1-cocycle for $G\curvearrowright \mathcal{U}$ satisfying

\begin{enumerate}
\item $\varphi(gh,e)=\varphi(g,h\cdot e)\varphi(h,e)$\quad (the 1-cocycle property),
 \item $\varphi(g,e)\cdot s(e)  \subseteq g\cdot s(e)$
\end{enumerate}
for all $g,h\in G$ and $e\in \mathcal{U}^1.$
\end{enumerate}
\end{defn}

\begin{rem}\label{rem3.3}
Condition (3)(b) above is necessary for making cycle change of the group action by the map $\varphi$ in the definition
$$g \cdot (ef):=(g\cdot e)(\varphi(g,e)\cdot f)\qquad (r(f)\subseteq s(e)).$$
In the case of \cite{bro14} for self-similar graphs, since the sources are also singletons, this condition would be $\varphi(g,e)\cdot s(e)=g \cdot s(e)$ for $g\in G$ and $e\in E^1$. However, we may have (even finite) self-similar ultragraphs which is not considered in \cite{bro14} (see Subsection \ref{sub4.1}).
\end{rem}

\begin{example}\label{ex3.4}
Let $G=\mathbb{Z}$ and let $\mathcal{U}=( U^0, {\mathcal{U}}^1 , r,s )$ be the ultragraph

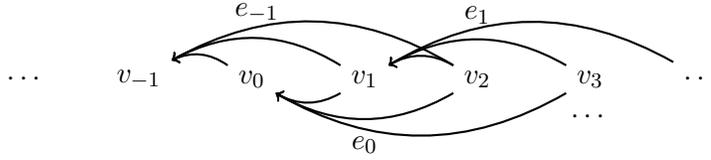
\begin{center}
\begin{tikzpicture}[thick]
\node (...) at (-1.5,0) {$\cdots$};
\node (v_{-1}) at (0,0) {$v_{-1}$};
\node (v_0) at (1.5,0) {$v_0$};
\node (v_1) at (3,0) {$v_1$};
\node (v_2) at (4.5,0) {$v_2$};
\node (v_3) at (6,0) {$v_3$};

\node (00) at (7.5,0) {$\cdots$};
\node (0) at (6,-0.5) {$\cdots$};

\path[->] (v_0) edge[bend right] node[right=1pt]  {$$} (v_{-1});
\path[->] (v_1) edge[bend right] node[above=2pt]  {$e_{-1}$} (v_{-1});
\path[->] (v_2) edge[bend right]   (v_{-1});
\path[->] (v_2) edge[bend right] node[right=1pt]  {$$} (v_1);
\path[->] (v_1) edge[bend left] node[left=1pt]  {$$} (v_0);
\path[->] (v_2) edge[bend left] node[above=-17pt]  {$e_0$} (v_0);
\path[->] (v_3) edge[bend right] node[above=2pt]  {$e_1$} (v_1);
\path[->] (00) edge[bend right] node[above=2pt]  {$$} (v_1);
\path[->] (v_3) edge[bend left] (v_0);
\end{tikzpicture}
\end{center}
where $ U^0=\{v_i \;:\; i\in \mathbb{Z}\}$ and ${\mathcal{U}}^1=\{e_i \;:\;  i\in \mathbb{Z}\}$ such that for each $i\in \mathbb{Z}$,
$$r(e_i)=\{v_i\} \;\;\text{     and     }\; \;s(e_i)=\{v_j \;:\; \;j>  i\}.$$
If the action $\mathbb{Z}\curvearrowright {\mathcal{U}}$ defined by
$$n \cdot v_i=v_{i+n} \;\;\text{     and     }\; \;n \cdot e_i=e_{i+n}\;\; (\forall \; i,n \in \mathbb{Z}),$$
and moreover $\varphi:  \mathbb{Z}\times {\mathcal{U}}^1 \rightarrow \mathbb{Z}$ is a $1$-cocycle satisfying
$$\varphi(n,e_i)\cdot s(e_i) \subseteq n \cdot s(e_i)\;\; (\forall \; i,n \in \mathbb{Z})$$
(for example the identity one $\varphi(n,e_i)=n)$, then $(\mathbb{Z},{\mathcal{U}},\varphi  )$ is a self-similar ultragraph.
\end{example}

\begin{example}\label{ex3.5}
Let ${\mathcal{U}}$ be the ultragraph

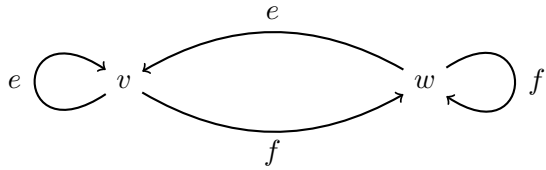
\begin{center}
\begin{tikzpicture}[thick]
\node (v) at (-1.5,0) {$v$};
\node (w) at (2.5,0) {$w$};

\path[<-] (v) edge[bend left] node [above=1pt]{$e$} (w);
\path[<-] (w) edge[bend left] node [above=-17pt]{$f$} (v);

\draw[<-] (w) ..controls (4,-1) and (4,1) .. node[right=1pt] {$f$} (w);

\draw[->] (v) ..controls (-3,-1) and (-3,1) .. node[left=1pt] {$e$} (v);
\end{tikzpicture}
\end{center}
and $G=\mathbb{Z}_{2}=\{0,1\}$. Define the action $\mathbb{Z}_{2}$ on ${\mathcal{U}}$ by $0 \cdot \alpha=\alpha$ for all $\alpha \in U^0 \cup {\mathcal{U}}^1$ and
$$1\cdot v=w , \;\;\; \;\;\;\;1\cdot w=v,$$
$$1\cdot e=f \;\;\text{and} \;\; 1\cdot f=e.$$
If $\varphi:  \mathbb{Z}_2 \times {\mathcal{U}}^1 \rightarrow \mathbb{Z}_2$ is a $1$-cocycle, then $(\mathbb{Z}_2,\mathcal{U},\varphi)$ would be a self-similar ultragraph.
\end{example}

If $(G,\mathcal{U},\varphi)$ is a self-similar ultragraph, we may extend the action $G\curvearrowright \mathcal{U}$ and the cocycle $\varphi$ on all ${\mathcal{U}}^*$ and also ${\mathcal{U}}^{\infty}$  as follows. First, for every $g\in{G}$ and $A\in{{\mathcal{U}}^0}$, set $\varphi(g,A)=g$. Thus, given any $\alpha=\alpha_1\alpha_2\in{\mathcal{U}}^*$ and $g\in {G}$, we define inductively
$$g\cdot \alpha:=(g\cdot\alpha_1)(\varphi(g,\alpha_1)\cdot \alpha_2)$$
and
$$\varphi(g,\alpha):=\varphi(\varphi(g,\alpha_1),\alpha_2).$$
In the same way, for every  $g\in{G}$ and $x=\alpha_1\alpha_2\cdots \in{\mathcal{U}}^{\infty}$ we can define
$$g\cdot x:=(g\cdot\alpha_1)(\varphi(g,\alpha_1)\cdot\alpha_2)(\varphi(g,\alpha_1\alpha_2)\cdot\alpha_3)\cdots$$
which again belongs to ${\mathcal{U}}^{\infty}$ (see \cite[Proposition 2.5]{bro14}). As pointed in Remark \ref{rem3.3},  Definition \ref{defn3.2}(3) insures that this definition are well-defined. Another useful property of $\varphi$  is the  following.

\begin{lem}\label{lem3.3}
For every $g\in G$ and $\alpha\in{\mathcal{U}}^*$, one has
$$\varphi(g^{-1},\alpha)=\varphi(g,g^{-1}\cdot \alpha)^{-1}.$$
\end{lem}

\begin{proof}
It follows immediately from
$$1_G=\varphi(1_G,\alpha)=\varphi(gg^{-1}, \alpha)=\varphi(g,g^{-1} \cdot \alpha)\varphi(g^{-1}, \alpha).$$
\end{proof}

Next, we associate a $C^*$-algebra to a self-similar ultragraph.

\begin{defn}\label{defn3.4}
Let $(G,\mathcal{U},\varphi)$ be a self-similar ultragraph as in Definition $\ref{defn3.2}$. A \textit{$(G,\mathcal{U})$-\textit{family}} is a collection of the form
$$\{p_{A},s_e:A\in {\mathcal{U}}^0,\; e\in {\mathcal{U}}^1\}\cup \{u_{A,g} :A\in {\mathcal{U}}^0, g\in G\}  $$
in a $C^*$-algebra  satisfying the following relations:

\begin{enumerate}
  \item $\{p_{A},s_e:A\in {\mathcal{U}}^0,\; e\in {\mathcal{U}}^1\}$ is a Cuntz-Krieger $\mathcal{U}$-family,
  \item $u_{A,{{1}_G}}=p_A$ for all $A\in {\mathcal{U}}^0$,
  \item $(u_{A,g})^*=u_{{{g^{-1}}\cdot A},{g^{-1}}}$ for all $A\in {\mathcal{U}}^0$ and $g\in G$,
  \item $(u_{A,g})(u_{B,h})=u_{{A\cap (g\cdot B)},{gh}}$ for all $A,B\in {\mathcal{U}}^0$ and $g,h\in G$,
  \item $(u_{A,g})s_e=\left\{
                        \begin{array}{ll}
                          s_{g\cdot e}u_{{g \cdot s(e)},{\varphi(g,e)}} & \text{if } g \cdot r(e)\subseteq A \\
                          0 & \text{otherwise}
                        \end{array}
                      \right.$
    for all $A\in {\mathcal{U}}^0, e\in {\mathcal{U}}^1$ and $g\in G$.
\end{enumerate}
The \emph{$C^*$-algebra ${\mathcal{O}}_{G,\mathcal{U}}$ associated to} $(G,\mathcal{U},\varphi)$ is the universal $C^*$-algebra generated by a $(G,\mathcal{U})$-family $\{p_A,s_e,u_{A,g}\}$. Corollary \ref{prop7.5} below insures that such $C^*$-algebra ${\mathcal{O}}_{G,\mathcal{U}}$ with nonzero generators $\{p_A,s_e,u_{A,g}\}$ exists.
\end{defn}

Note that items $(2)$ and $(4)$ above in the case $h=1_G$ imply that
$$u_{A,g}  p_B=u_{A\cap(g\cdot B),g}.$$

\begin{prop}\label{prop3.5}
Let $(G,\mathcal{U},\varphi)$ be a self-similar ultragraph as in Definition $\ref{defn3.2}$. Then
\begin{equation}\label{(3.1)}
{\mathcal{O}}_{G,\mathcal{U}}=\mathrm{\overline{span}}\{s_{\alpha}u_{A,g} s^*_{\beta}\;:g\in G,\; \alpha,\beta \in{\mathcal{U}}^*,\;A\in {\mathcal{U}}^0\; and\; A\subseteq{s(\alpha)\cap g\cdot s(\beta})\},
\end{equation}
where $s_\alpha:=p_A$ if $\alpha=A\in \mathcal{U}^0$, and $s_\alpha:=s_{e_1} \ldots s_{e_n}$ if $\alpha=e_1 \ldots e_n \in \mathcal{U}^n.$
\end{prop}

\begin{proof}
Letting $\mathcal{M}=\text{span}\{s_{\alpha}u_{A,g}s^*_{\beta} : g\in G, \alpha,\beta \in{\mathcal{U}}^*, A\in {\mathcal{U}}^0\; \text{and}\; A\subseteq{s(\alpha)\cap g\cdot s(\beta})\}$, we show that $\mathcal{M}$ is a $*$-subalgebra of ${\mathcal{O}}_{G,\mathcal{U}}$. In order to do this, fix two elements $s_{\alpha}u_{A,g} s^*_{\beta}$ and $s_{\gamma}u_{B,h} s^*_{\delta}$ in $\mathcal{M}$. In the case $\gamma=\beta \varepsilon$ for some $\varepsilon \in {\mathcal{U}}^{\geq 1}$ with $r(\varepsilon)\subseteq g^{-1} \cdot A$, Definition \ref{defn3.4} implies that
\begin{align}\label{(3.2)}
(s_{\alpha}u_{A,g} s^*_{\beta})(s_{\gamma}u_{B,h} s^*_{\delta})&=s_{\alpha}u_{A,g} (s^*_{\beta}s_{\beta})s_{\varepsilon}u_{B,h} s^*_{\delta}\\
&=s_{\alpha}(u_{A,g}s_{\varepsilon})u_{B,h} s^*_{\delta}\nonumber\\
&=(s_{\alpha}s_{g\cdot \varepsilon })[(u_{{g\cdot s(\varepsilon)} ,{{\varphi(g,\varepsilon)}}}) (u_{B,h} )]s^*_{\delta}\nonumber \\
&=s_{\alpha(g\cdot\varepsilon)}\;u_{{(g\cdot s(\gamma))\cap(\varphi(g,\varepsilon)\cdot B ) },{ {\varphi(g,\varepsilon)}h}}\;{s^*_{\delta}}\;\; \;\;\;\;\;\;\; (\text{as}\;\; g\cdot r(\varepsilon)\subseteq A).\nonumber
\end{align}
Since
\begin{align}
(g\cdot s(\varepsilon))\cap(\varphi(g,\varepsilon)\cdot B) &\subseteq( g\cdot s(\varepsilon))\cap(\varphi(g,\varepsilon)h \cdot   s(\delta))\nonumber\\
&= s( g\cdot\varepsilon)\cap(\varphi(g,\varepsilon) h \cdot s(\delta))\nonumber \\
&=s( \alpha (g\cdot\varepsilon))\cap(\varphi(g,\varepsilon) h \cdot s(\delta)),\nonumber
\end{align}
it follows $(s_{\alpha}u_{A,g} s^*_{\beta})(s_{\gamma}u_{B,h} s^*_{\delta})\in {\mathcal{M}}$.

Multiplication (\ref{(3.2)}) for the cases $\beta=\gamma$ and $\beta=\gamma\varepsilon$ with $r(\varepsilon)\subseteq B$ are analogous, while for others is zero. Moreover, we have
$$(s_{\alpha}u_{A,g} s^*_{\beta})^*=s_{\beta}\;u_{{g^{-1}\cdot A} ,{g^{-1}}} \;s^*_{\alpha},$$
so $\mathcal{M}$ is self-adjoint and closed under multiplication. Therefore, $\overline{\mathcal{M}}$ is a closed $*$-subalgebra of ${\mathcal{O}}_{G,\mathcal{U}}$ containing the generators of ${\mathcal{O}}_{G,\mathcal{U}}$, whence $\overline{\mathcal{M}}={\mathcal{O}}_{G,\mathcal{U}}$.
\end{proof}


{\section{examples}\label{sec4}

Before giving inverse semigroup and groupoid models for self-similar ultragraph $C^*$-algebras, in this section we indicate some class of examples being of the form of the $C^*$-algebras of self-similar ultragraphs.

In a self-similar ultragraph $(G,\mathcal{U},\varphi)$, if the source edges are also singletons, then $(G,\mathcal{U},\varphi)$ would be trivially a self-similar graph in the sense of \cite{bro14}. In Subsection \ref{sub4.1}, we show moreover that a self-similar ultragraph $C^*$-algebra over a ``source-finite" ultragraph $\mathcal{U}$ (with a mild condition) is isomorphic to an Exel-Pardo $C^*$-algebra \cite{bro14}. However, we may have a triple $(G,\mathcal{U},\varphi)$ over a finite ultragraph $\mathcal{U}$ which could not presented via a self-similar graphs  \cite{bro14}. See Remark \ref{rem3.3}.

Furthermore, in Subsection \ref{sub4.2}, we simulate the self-similar group actions on graphs and ultragraphs for square $\{0,1\}$-matrices and Exel-Laca algebras.$\;$Recall that, Exel and Laca in \cite{li18}  introduced another extension of the Cuntz-Krieger algebras, besides the graph $C^*$-algebras, which are generated by infinite square $\{0,1\}$-matrices. In light of \cite[Theorem 4.5]{ana00} (see also \cite[Remark 4.6]{ana00}), these $C^*$-algebras are included in the class of ultragraph $C^*$-algebras. Hence, it is a natural question that  ``how may we define a self-similar action of a group $G$ on a $\{0,1\}$-matrix $M$ and its $C^*$-algebra according to Definitions \ref{defn3.2} and \ref{defn3.4}?'' Subsection \ref{sub4.2} will be devoted to answer to this question, where we will define an $M$-invariant self-similar triple $(G,I,\varphi)$ and associated $C^*$-algebra.

\subsection{Source-finite ultragraphs}\label{sub4.1}

Fixed a self-similar ultragraph $(G,\mathcal{U},\varphi)$, we associate a self-similar graph $(G,E_\mathcal{U},\phi)$ as follows. Let $E_{\mathcal{U}}^0:=U^0$ be as the vertex set. For each edge $e\in{\mathcal{U}}^1$ and $v\in s(e)$, draw an edge $e^v$  in $E_{\mathcal{U}}$ from $v$ to $r(e)$. Then the edge set of $E_{\mathcal{U}}$ is
$$E_{\mathcal{U}}^1=\{e^v\;:\; e\in {\mathcal{U}}^1 \; \text{and} \; v\in s(e)\}$$
such that $r_{E}(e^v)=r(e)$ and $s_{E}(e^v)=v$ for all $e^v\in E_{\mathcal{U}}^1$. For example

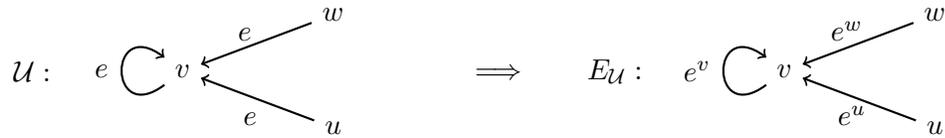
\begin{center}
\begin{tikzpicture}[thick]
\node ({U}) at (-13,0) {$\mathcal{U}:$};
\node (v) at (-11,0) {$v$};
\node (w) at (-9,0.75) {$w$};
\node (u) at (-9,-0.75) {$u$};

\node (E_u) at (-6,0) {$\Longrightarrow \qquad E_\mathcal{U}:$};

\node (V) at (-3,0) {$v$};
\node (W) at (-1,0.75) {$w$};
\node (U) at (-1,-0.75) {$u$};

\path[<-] (v) edge node[above=-14pt] {$e\hspace{2mm}$} (u);
\path[<-] (v) edge node[above=-3pt] {$\hspace{-3mm}e$} (w);

\draw[->] (V) ..controls (-4,-0.75) and (-4,0.75) .. node[left=1pt] {$e^v$} (V);

\path[<-] (V) edge node[above=-14pt] {$e^u\hspace{2mm}$} (U);
\path[<-] (V) edge node[above=-3pt] {$\hspace{-3mm}e^w$} (W);

\draw[->] (v) ..controls (-12,-0.75) and (-12,0.75) .. node[left=1pt] {$e$} (v);
\end{tikzpicture}

\end{center}
Moreover, the action of $G$ on $E_{\mathcal{U}}^0$ is just that on $U^0$ from $(G,\mathcal{U},\varphi)$, while for each $e^v\in E_{\mathcal{U}}^1 $ we define
$$g\cdot e^v:=({g\cdot e})^{g\cdot v}.$$
If, furthermore, we set $\phi(g,e^v):=\varphi(g,e)$ for $g\in G$ and $e^v \in  E_{\mathcal{U}}^1$, then $(G,E_{\mathcal{U}},\varphi)$ is a self-similar graph by the next lemma.
\begin{lem}
Let $(G,\mathcal{U},\varphi)$ be a self-similar ultragraph such that $\varphi (g,e)\cdot v=g\cdot v$ for all $g\in G, e \in {\mathcal{U}}^1$ and $v\in s(e)$. Then its graph triple $(G,E_\mathcal{U},\phi)$ defined above is a self-similar graph in the sense of Exel-Pardo \cite{bro14}.
\end{lem}

\begin{proof}
For any $g\in G$ and $e^v \in  E_{\mathcal{U}}^1$, we have
$$s_{E}(g\cdot e^v)=s_{E}(({g\cdot e})^{g\cdot v})=g\cdot v=g\cdot s_{E}(e^v)$$
and
$$r_{E}(g\cdot e^v)=r_{E}(({g\cdot e})^{g\cdot v})=r(g\cdot e)=g\cdot r(e)=g\cdot r_{E}(e^v).$$
So, automorphisms $g:U^0 \sqcup {\mathcal{U}}^1      \rightarrow           U^0 \sqcup {\mathcal{U}}^1$ induce group automorphisms $g: E_{\mathcal{U}}^0 \sqcup  E_{\mathcal{U}}^1      \rightarrow            E_{\mathcal{U}}^0 \sqcup  E_{\mathcal{U}}^1  $ for all $g\in G$. Moreover $\phi:G\times E_{\mathcal{U}}^1  \rightarrow   G$ is a 1-cocycle because
\begin{align}
\phi(gh,e^v)&=\varphi(gh,e)\nonumber\\
&=\varphi(g,h\cdot e)\varphi(h,e)\nonumber\\
&=\phi(g,(h\cdot e)^{(h\cdot v)})\phi(h,e^v)\nonumber \\
&=\phi(g,h\cdot e^v)\phi(h,e^v)\nonumber
\end{align}
for all $g,h\in G$ and $e^v \in E_{\mathcal{U}}^1$. Therefore, $(G,E_\mathcal{U},\phi)$ is a self-similar graph.
\end{proof}
\begin{defn}
We say an ultragraph ${\mathcal{U}}=(U^0,{\mathcal{U}}^1,r,s)$ is {\it{source-finite}} whenever $|s(e)|< \infty$ for all $e\in {\mathcal{U}}^1$.
\end{defn}
Note that if $\mathcal{U}$ is a source-finite ultragraph, then ${\mathcal{U}}^0$ is precisely the set of all finite subsets of $U^0$.

\begin{prop}\label{prop4..3}
Let $(G,\mathcal{U},\varphi)$ be a self-similar ultragraph over a source-finite ultragraph $ {\mathcal{U}}$ and $(G,E_\mathcal{U},\phi)$ the associated self-similar graph. Suppose also $\varphi(g,e)\cdot v=g\cdot v$ for all $g\in G,\; e \in {\mathcal{U}}^1$ and $v\in s(e)$. Then ${\mathcal{O}}_{G,\mathcal{U}}\cong{\mathcal{O}}_{G,E_\mathcal{U}}$.
\end{prop}

\begin{proof}
Let $\{s_e,p_A,u_{A,g}\}$ and $\{t_{e^v},q_v,z_{v,g}\}$ be the generating $(G,\mathcal{U})$- and $({G,E_\mathcal{U}})$-families for ${\mathcal{O}}_{G,\mathcal{U}}$ and ${\mathcal{O}}_{G,E_\mathcal{U}}$ respectively. According to \cite[Remark 2.3]{bur81}, $z_g:=\sum_{v\in E_{\mathcal{U}}^0 }{z_{v,g}}$ is a unitary in the multiplier algebra ${\mathcal{M}}({\mathcal{O}}_{G,E_\mathcal{U}})$ for any $g\in G$, which gives a unitary $*$-representation $g \mapsto z_g$ of $G$ into ${\mathcal{M}}({\mathcal{O}}_{G,E_\mathcal{U}})$ satisfying
$$z_{v,g}=q_vz_g\qquad \text{and} \qquad z_g t_{e^v}=t_{g\cdot e^v}z_{\phi(g, e^v)}$$
for $v\in E_{\mathcal{U}}^0$ and $ e^v \in  E_{\mathcal{U}}^1 $. For every $A\in {\mathcal{U}}^0,\; g\in G$ and $e\in {\mathcal{U}}^1$, define
$$P_A:=\sum_{v\in A} q_v,\;\;\;\;U_{A,g}:=\sum_{v\in A} z_{v,g}=\sum_{v\in A}q_v z_g \qquad \text{and} \qquad S_e:=\sum_{v\in s(e)}t_{e^v}.$$
As noted before, since ${\mathcal{U}}$ is source-finite, every $A\in {\mathcal{U}}^0$ is finite, so the above summations are finite and $P_A, U_{A,g}$ and $S_e$ belong to ${\mathcal{O}}_{G,E_\mathcal{U}}$. It is easy to show that $\{S_e,P_A,U_{A,g}\}$ is a $(G,\mathcal{U})$-family by verifying relations in Definition \ref{defn3.4}; here, we verify relations (4) and (5) for example. For (4), letting arbitrary $A,B\in {\mathcal{U}}^0$ and $g,h\in G$, we have
\begin{align}
U_{A,g}U_{B,h}&=(\sum_{v\in A}q_v z_g)(\sum_{w\in B}q_w z_h)\nonumber\\
&=\sum_{v\in A}\sum_{w\in B}q_v (z_g q_w)z_h\nonumber\\
&=\sum_{v\in A}\sum_{w\in B}q_v(q_{g\cdot w}z_g)z_h\nonumber \\
&=\sum_{v\in {A\cap(g\cdot  B)}}q_v z_{gh}\nonumber\\
&=U_{{{A\cap(g\cdot  B)}},gh}.\nonumber
\end{align}
Similarly, for relation (5), one can write
\begin{align}
U_{A,g}S_e&=(\sum_{w\in A}q_w z_g)q_{r(e)} S_e\nonumber\\
&=(\sum_{w\in A}q_w q_{g\cdot r(e)}z_g)(\sum_{v\in s(e)}t_{e^v})\nonumber\\
&=\sum_{w\in A}\sum_{v\in s(e)}q_w q_{g\cdot r(e)}t_{g\cdot{e^v}}z_{\phi(g,e^v)}.\nonumber
\end{align}
So, if $g\cdot r(e) \subseteq A$, it is equal to $S_{g\cdot e}U_{g\cdot s(e),\varphi(g,e)}$ and otherwise is zero. Therefore, $\{S_e,P_A,U_{A,g}\}$ is a $(G,{\mathcal{U}})$-family in ${\mathcal{O}}_{G,\mathcal{U}}$. On the other hand, if we define
$$Q_v:=P_{\{v\}},\qquad T_{e^v}:=s_ep_{\{v\}}, \qquad \text{and} \qquad Z_{v,g}:=u_{\{g \cdot v\},g}$$
for every $v\in E_{\mathcal{U}}^0,\; e^v \in E_{\mathcal{U}}^1$ and $g\in G$, then $\{Q_v,T_{e^v},Z_{v,g}\}$ is a $(G,E_{\mathcal{U}})$-family in ${\mathcal{O}}_{G,\mathcal{U}}$. Thus, by universality, there are $*$-homomorphisms $\pi _1:{\mathcal{O}}_{G,\mathcal{U}}     \rightarrow   {\mathcal{O}}_{G,E_{\mathcal{U}}}$ and $\pi _2:{\mathcal{O}}_{G,\mathcal{U}}     \rightarrow   {\mathcal{O}}_{G,E_{\mathcal{U}}}$ such that $\pi _1 \circ \pi _2$ and $\pi _2 \circ \pi _1$ are identity on the generators of ${\mathcal{O}}_{G,E_{\mathcal{U}}}$ and ${\mathcal{O}}_{G,\mathcal{U}} $  respectively, and hence $\pi _1$ and $\pi _2$ are isomorphisms. The proof is complete.
\end{proof}

Furthermore, we should note that the condition $\varphi (g,e)\cdot v=g\cdot v$ for $v\in s(e)$, in Proposition \ref{prop4..3} is stronger than the one $\varphi (g,e)\cdot s(e)\subseteq g\cdot s(e)$ in Definition \ref{defn3.2}(3)(b) for general self-similar ultragraphs. So, we may construct a triple $(G,\mathcal{U},\varphi)$ over a source-finite ultragraph $\mathcal{U}$ such that its graph triple $(G,E_\mathcal{U},\phi)$ is not a self-similar graph in the sense of  \cite{bro14}. For example, let $G={\mathbb{Z}}_2=\{0,1\}$ and let $\mathcal{U}$ be the finite ultragraph

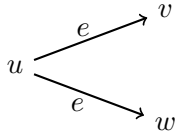
\begin{center}
\begin{tikzpicture}[thick]
\node (u) at (-11,0) {$u$};
\node (v) at (-9,0.75) {$v$};
\node (w) at (-9,-0.75) {$w$};

\path[->] (u) edge node[above=-3pt] {$e\hspace{2mm}$} (v);
\path[->] (u) edge node[above=-10pt] {$\hspace{-3mm}e$} (w);

\end{tikzpicture}
\end{center}
with the action ${\mathbb{Z}}_2  \curvearrowright  \mathcal{U}$ defined by
$$1 \cdot v=w, \qquad  1\cdot w=v,$$
$$1\cdot u=u, \qquad 1\cdot e=e,$$
and $0 \cdot a=a$ for all $a\in \{  u,v,w,e     \}$. If  $\varphi:{\mathbb{Z}}_2 \times {\mathcal{U}}^1   \rightarrow {\mathbb{Z}}_2$ is the zero cocycle, that is $\varphi(g,e)=0$ for $g\in  {\mathbb{Z}}_2$, then $( {\mathbb{Z}}_2,  {\mathcal{U}},\phi      )$ is a self-similar ultragraph as in Definition \ref{defn3.2}. But, since
$$\phi  (1,e^v)\cdot v=0 \cdot v=v \neq 1\cdot v,$$
the graph triple $( {\mathbb{Z}}_2,  E_{\mathcal{U}},\phi)$ is not a self-similar graph in the sense of \cite{bro14}. Therefore, the concept of self-similar ultragraphs even for finite ultragraphs is more general than that of the Exel-Pardo self-similar graphs.

\subsection{Self-similar group actions on the Exel-Laca algebras}\label{sub4.2}

Exel and Laca in \cite{li18} generalized the Cuntz-Krieger algebras for infinite square $\{0,1\}$-matrices. Inspired from the obtained results for the graph $C^*$-algebras, in particular in \cite{li19,nek04}, they then studied these $C^*$-algebras. Note that the Exel-Laca algebras for row-finite $\{0,1\}$-matrices with no identically zero rows coincide with the graph $C^*$-algebras of row-finite, source-free graphs.

Mark Tomforde in \cite[Section 4]{ana00} showed that every Exel-Laca algebra is an ultragraph $C^*$-algebra such that in the associated ultragraph every vertex receives exactly one edge. However, the class of ultragraph $C^*$-algebras is substantially larger than that of graph $C^*$-algebras plus Exel-Laca algebras.

Let $G$ be a countable discrete group and let $M$ be a countable square $\{0,1\}$-matrix. Here, using Definition \ref{defn3.4}, we define a specific group action of $G$ on $M$ and associated $C^*$-algebra ${\mathcal{O}}_{G,I,M}$. We will then show that ${\mathcal{O}}_{G,I,M}$ is isomorphic to a self-similar ultragraph $C^*$-algebra.

Let $I$ be a countable index set and $M=(a_{i,j})$ an $I\times I$ $\{0,1\}$-matrix. We will denote the $ i^{th}$ row of $M$ by $M_i=(a_{i,j})_{j\in I}$ in the sequel. An {\it{M-invariant} action} of $G$ on $I$ is a collection of bijections $g:I\rightarrow I$, denoted by $g\cdot i:=g(i)$, satisfying

\begin{enumerate}

  \item $g \circ g^{-1}=g^{-1} \circ g=i d_I$ and
  \item $M_{g\cdot i}=(a_{g\cdot i,g\cdot j})_{j\in I}$ for all $i\in I$ (in other words, $a_{g\cdot i,g\cdot j}=a_{g\cdot i,j}$ for all $g\in G$ and $i,j\in I$).
\end{enumerate}

\begin{defn}\label{defn4.6}
Let $M$ be an $I\times I$ $\{0,1\}$-matrix with no identically zero rows. By an {\it{$M$-invariant self-similar triple}}, we mean a triple $(G,I,\varphi)$ such that

\begin{enumerate}

  \item there is an $M$-invariant action $G \curvearrowright I$, and
  \item $\varphi:G\times I\rightarrow G$ is a map satisfying

\begin{enumerate}

  \item (the $1$-cocycle property:) $\varphi(gh,i)=\varphi(g,h \cdot i)\varphi(h,i)$ for every $g,h \in G$ and $i\in I$, and
  \item for each $i\in I,\; j\in \{k\in I\;:\; a_{i,k}=1\}$ and $g\in G$, there is $j'\in \{k\in I \;:\; a_{i,k}=1\}$ such that $\varphi(g,i)\cdot j=g\cdot j'.$
\end{enumerate}
\end{enumerate}
\end{defn}

We now associate a $C^*$-algebra to an $M$-invariant triple $(G,I,\varphi)$.

\begin{defn}\label{defn4.7}
Let $M$ be an $I\times I$ $\{0,1\}$-matrix with no identically zero rows and $(G,I,\varphi)$ an $M$-invariant triple. A {\it{$(G,I)$-family}} is a collection $\{s_i: i\in I\}\cup \{u_g:g\in G\}$ in a unital $C^*$-algebra $\mathfrak{B}$ satisfying

\begin{enumerate}

  \item $\{s_i\;:\; i\in I\}$ is an Exel-Laca family in the sense of \cite[Definition 7.1]{li18},
  \item the map $g \mapsto u_g$ is a unitary $*$-representation of $G$ on $\mathfrak{B}$, and moreover
  \item $u_g s_i=s_{g\cdot i}u_{\varphi(g,i)}$ for all $g\in G$ and $i\in I$.
\end{enumerate}
Then the $C^*$-algebra $\mathcal{O}_{G,I,M}$ is defined by
$$\mathcal{O}_{G,I,M}=\mathrm{\overline{span}}\{s_i u_g s^{*}_j\; :\; i,j \in I\; \text{and}\; g\in G\}$$
where $\{s_i,u_g\}$ is a universal $(G,I)$-family in a unital $C^*$-algebra.
\end{defn}

Our aim is to construct a self-similar ultragraph $(G,{\mathcal{U}}_M,\varphi)$ for which the $C^*$-algebra ${\mathcal{O}}_{G,I,M}$ is isomorphic to ${\mathcal{O}}_{G,{\mathcal{U}}_M}$.

\begin{lem}\label{lem4.8}
If $(G,I,\varphi)$ is an $M$-invariant triple, then
$$\varphi(g,i)^{-1}=\varphi(g^{-1},g\cdot i)$$
for all $g\in G$ and $i\in I$.
\end{lem}
\begin{proof}
The proof is analogous to that of Lemma $\ref{lem3.3}$.
\end{proof}

\begin{lem}
If $\{s_i,u_g\}$ is a $(G,I)$-family as in Definition $\ref{defn4.7}$, then
\begin{equation}\label{(4.1)}
u_g s_i s^{*}_i=s_{g\cdot i}s^{*}_{g\cdot i} u_g
\end{equation}
for all $g\in G$ and $i\in I$.
\end{lem}

\begin{proof}
By the relations in Definition $\ref{defn4.7}$, we can write
\begin{align}
u_g s_i s^{*}_i&=s_{g\cdot i}u_{\varphi(g,i)}s^{*}_i\nonumber\\
&=s_{g\cdot i}{(   s_i  u_{\varphi(g,i)^{-1}} )}^*\nonumber\\
&=s_{g\cdot i}{( s_{g^{-1}g\cdot i} u_{\varphi(g^{-1},g\cdot i)} )}^*\qquad \text{(by Lemma }\ref{lem4.8})    \nonumber\\
&=s_{g\cdot i}(u_{g^{-1}}s_{g\cdot i})^*\nonumber\\
&=s_{g\cdot i}s^{*}_{g\cdot i}u_g,\nonumber
\end{align}
establishing ($\ref{(4.1)}$).
\end{proof}

Let $(G,I,\varphi)$ be an $M$-invariant triple. By \cite[Section 4]{ana00}, one may associate the ultragraph ${\mathcal{U}}_M$ to $(I,M)$ such that $M$ is the edge matrix of ${\mathcal{U}}_M$; that is $U^0_M=\{v_i\;:\; i\in I\}$, ${\mathcal{U}}^{1}_M=I$, the range and source maps are
$$r(i)=\{v_i\}\qquad \text{and} \qquad s(i)=\{v_j\;:\: a_{i,j}=1\}$$
for all $i\in I$. Moreover, if we define $g\cdot v_i :=v_{g\cdot i}$ for $i\in I$, also $g\cdot i$ and $\varphi(g,i)$ are as those in $(G,I,\varphi)$, then $(G,{\mathcal{U}}_M,\varphi)$ is a self-similar ultragraph. The next proposition shows that ${\mathcal{O}}_{G,I,M}$ is isometrically isomorphic to ${\mathcal{O}}_{G,{\mathcal{U}}_M}$.$\;$Note that since $M$ has no identically zero rows, every vertex $v_i$ in the ultragraph ${\mathcal{U}}_M$ receives the unique edge $i$.

\begin{prop}\label{prop4.10}
Let $M$ be an $I \times I$, $\{0,1\}$-matrix with no identically zero rows and $(G,I,\varphi)$ an $M$-invariant triple as in Definition $\ref{defn4.6}$. If $(G,{\mathcal{U}}_M,\varphi)$ is the associated self-similar ultragraph as above, then ${\mathcal{O}}_{G,I,M}\cong {\mathcal{O}}_{G,{\mathcal{U}}_M}.$

\end{prop}

\begin{proof}
Suppose that $\{s_i,p_i,u_{A,g}\}$ and $\{t_i,w_g\}$ are the generators of ${\mathcal{O}}_{G,{\mathcal{U}}_M}$ and ${\mathcal{O}}_{G,I,M}$, respectively. According to the relations in Definition $\ref{defn3.4}$, $u_{A,g}$'s are partial unitaries, so $\sum_{v\in{ U^{0}_M}}u_{\{v\},g}$ converges to a unitary $u_g\in M({\mathcal{O}}_{G,{\mathcal{U}}_M})$ in the strong topology satisfying
$$p_A u_g=u_{A,g}$$
for all $A\in {\mathcal{U}}^0 _A$. Moreover, \cite[Proposition 4.4]{ana00} says that $\{s_i : i\in I\}$ is an Exel-Laca family  generating ${\mathcal{O}}_M$, and thus $\{s_i,u_g\}$ satisfies relations in Definition $\ref{defn4.7}$. Hence the universality of ${\mathcal{O}}_{G,I,M}$ gives a $*$-homomorphism $\pi:{\mathcal{O}}_{G,I,M}\rightarrow M({\mathcal{O}}_{G,{\mathcal{U}}_M})$ such that $\pi(t_i)=s_i$ and $\pi(w_g)=u_g$ for each $i\in I$ and $g\in G$.

 On the other hand, define $P_{s(i)}=s^* _i s_i$ and $P_{r(i)}=s_i s^* _i$ for every $i\in I$. Recall that $M$ has no identically zero rows, so the ultragraph ${\mathcal{U}}_M$ has no sources. Moreover, each $A\in{\mathcal{U}}^0 _M $ is produced by elements in $\{r(i),s(i): i\in I\}$ via finite operations $\cup,\cap$ and $\setminus$. Hence the corresponding operations $P_{A \cap B}:=P_A P_B$, $P_{A\cup B}:=P_A+P_B-P_{A \cap B}$, and $P_{A\setminus B}=P_A- P_{A \cap B}$, generate projections $P_A$ for all $A\in{\mathcal{U}}^0 _M $. Note that since $\{P_{s(i)},P_{r(i)}\;:\; i\in I\}$ are commutative, the definition of $P_A$, for each $A\in {\mathcal{U}}^0 _M$, is well-defined. Now, by defining $S_i:=t_i$ and $U_{A,g}:=P_Aw_g$ for $i\in I$, $A\in{\mathcal{U}}^0 _M$ and $g\in G$, it is routine to verify that $\{S_i,P_A,U_{A,g}\}$ is a $(G,{\mathcal{U}} _M)$-family in ${\mathcal{O}}_{G,I,M}$ and the universality gives again a $*$-homomorphism $\psi:{\mathcal{O}}_{G,{\mathcal{U}}_M}\rightarrow {\mathcal{O}}_{G,I,M}$ such that $\psi$ sends $s_i,p_i,u_{A,g}$ to $S_i,P_i$ and $U_{A,g}$ respectively. Since $\pi \circ \psi$ and $\psi \circ \pi$ are identity on the generators, $\pi$ and $\psi$ are the inverse of each other and we conclude that ${\mathcal{O}}_{G,I,M}\cong{\mathcal{O}}_{G,{\mathcal{U}}_M}$.
\end{proof}

\subsection{A specific self-similar ultragraph associated a self-similar graph}

In the ultragraph ${\mathcal{U}}_{M}$ considered in Proposition  \ref{prop4.10}, every vertex recieves exactly one edge. Conversely, if we have a self-similar ultragraph $(G,{\mathcal{U},\varphi})$ satisfying $|r^{-1}(\{v\})| =1$ for all vertex $v \in U^0$ (such as those in Examples \ref{ex3.4} and \ref{ex3.5}), then we may construct an $M$-invariant triple $(G,I,\varphi)$ such that ${\mathcal{O}}_{G,\mathcal{U}}$ is isomorphic to the $C^*$-algebra ${\mathcal{O}}_{G,I,M}$ in Definition \ref{defn4.7}. (In this case $I=U^0$ and $M$ is the edge matrix of ${\mathcal{U}}$.) In this subsection, we show that every self-similar graph $C^*$-algebra ${\mathcal{O}}_{G,E}$ of \cite{bro14} with a row-finite source-free graph $E$  belongs to this class of self-similar ultragraph $C^*$-algebras (so it will be also of the form of the $C^*$-algebras defined in Definition \ref{defn4.7}).

Let $(G,E,\phi)$ be a  self-similar graph over a row-finite graph $E$ without sources. We want to construct a self-similar ultragraph $(G,{\mathcal{U}}_E,\varphi)$ such that

\begin{enumerate}

  \item every singleton set $\{v\}$ in ${\mathcal{U}}^0$ receives only one edge, and
  \item ${\mathcal{O}}_{G,E_{\mathcal{U}}}  \cong {\mathcal{O}}_{G,{\mathcal{U}}_E}$.
\end{enumerate}

To this end, we decompose every vertex $v\in E^0$ according to edges in $r^{-1}(v)=\{e \in {\mathcal{U}} ^1\; :\; r(e)=v\}$ by putting vertex $v_e$ for each $e\in r^{-1}(v)$. So,
$${U^0_E} =\bigcup _{v\in E^0}\{v_e\; :\; e \in r^{-1}(v)\}=\{v_e\; :\; v\in E^0\; \text {and} \;e \in r^{-1}(v)\}$$
and the edge set is just ${\mathcal{U}}^{1}_E=E^1$, where $s_{\mathcal{U}}(e)=\{s(e)_f \;:\; f\in r^{-1}(s(e))\}$ and $r_{\mathcal{U}}(e)=\{r(e)_e\} $. In fact, $s_{\mathcal{U}}(e)$ is set of all decomposed components from $s(e)$, while $r_{\mathcal{U}}(e)$ is the singleton $\{r(e)_e\} $.

\begin{example}
If $E$ be the graph

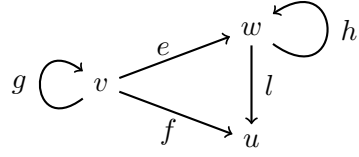
\begin{center}
\begin{tikzpicture}[thick]

\node (v) at (-11,0) {$v$};
\node (w) at (-9,0.75) {$w$};
\node (u) at (-9,-0.75) {$u$};

\path[<-] (u) edge node[above=-16pt] {$f\hspace{2mm}$} (v);
\path[<-] (w) edge node[above=-3pt] {$\hspace{-3mm}e$} (v);

\path[<-] (u) edge node[right=1pt] {$l\hspace{2mm}$} (w);

\draw[->] (v) ..controls (-12,-0.75) and (-12,0.75) .. node[left=1pt] {$g$} (v);

\draw[->] (w) ..controls (-7.75,-0.15) and (-7.75,1.6) .. node[right=1pt] {$h$} (w);
\end{tikzpicture}

\end{center}
then  $r^{-1}(v)=\{g\},\; r^{-1}(w)=\{e,h\}$, and $r^{-1}(u)=\{f,l\}$. So, ${U^0_E}=\{v_g, w_e, w_h ,u_f,u_l\}$ and ${\mathcal{U}}^{1}_E=\{g,e,f,h,l\}$ where
$$s_{\mathcal{U}}(g)=r_{\mathcal{U}}(g)=\{v_g\}$$
$$s_{\mathcal{U}}(e)=\{v_g\}\;,\;r_{\mathcal{U}}(e)=\{w_g\}$$
$$s_{\mathcal{U}}(f)=\{v_g\}\;,\;r_{\mathcal{U}}(f)=\{u_g\}$$
$$s_{\mathcal{U}}(h)=\{w_e,w_h\}\;,\;r_{\mathcal{U}}(h)=\{w_h\}$$
$$s_{\mathcal{U}}(l)=\{w_e,w_h\}\;,\;r_{\mathcal{U}}(l)=\{u_l\}.$$
Hence the ultragraph ${\mathcal{U}}_E$ would be

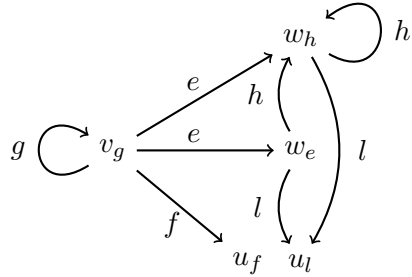
\begin{center}
\begin{tikzpicture}[thick]

\node (v_g) at (-11.5,0) {$v_g$};
\node (w_h) at (-9,1.5) {$w_h$};
\node (w_e) at (-9,0) {$w_e$};
\node (u_l) at (-9,-1.5) {$u_l$};
\node (u_f) at (-9.7,-1.5) {$u_f$};

\path[<-] (w_h) edge node[above=-2pt] {$\hspace{-3mm}e$} (v_g);
\path[<-] (w_e) edge node[above=-1pt] {$\hspace{-3mm}e$} (v_g);

\path[<-] (w_h) edge[bend right] node[right=6pt] {$\hspace{-11mm}h$} (w_e);

\path[<-] (u_l) edge[bend left] node[left=2pt] {$\hspace{1mm}l$} (w_e);
\path[<-] (u_l) edge[bend right] node[right=6pt] {$\hspace{-1mm}l$} (w_h);
\path[<-] (u_f) edge node[above=-16pt] {$f\hspace{2mm}$} (v_g);

\draw[->] (v_g) ..controls (-12.7,-0.75) and (-12.7,0.75) .. node[left=1pt] {$g$} (v_g);

\draw[->] (w_h) ..controls (-7.70,0.75) and (-7.70,2.5) .. node[right=1pt] {$h$} (w_h);

\end{tikzpicture}

\end{center}

Note that the ultragraph ${\mathcal{U}}_E$ is row-finite and source-free as the graph $E$ is.
Furthermore, we may define $\varphi(g,e):=\phi(g,e)$ for every $e\in {\mathcal{U}}_{E}^1$ and $g\in G$, while $g\cdot v_e:=v_{g\cdot e}$, ${g\cdot e}$ are defined as in $(G,E,\phi)$. So we get:
\end{example}
\begin{prop}
Let $(G,E,\phi)$ be a self-similar graph over a row-finite, source-free graph $E$. Then
\begin{enumerate}

  \item The triple $(G,{\mathcal{U}}_E,\varphi)$ defined above is a self-similar ultragraph such that each singleton set $\{v\}\in {\mathcal{U}}_{E}^0$ receives exactly one edge.
  \item The $C^*$-algebra ${\mathcal{O}}_{G,E}$ is isomorphic to ${\mathcal{O}}_{G,{\mathcal{U}}_E}$.
\end{enumerate}
\end{prop}


\section{An inverse semigroup associated to $(G,\mathcal U,\varphi)$}\label{sec5}

Let $(G,\mathcal{U},\varphi)$ be a self-similar ultragraph. Inspired by Proposition \ref{prop3.5}, we define a specific inverse semigroup ${\mathcal{S}}_{G,\mathcal{U}}$ so that the canonical map $\pi:{\mathcal{S}}_{G,\mathcal{U}} \rightarrow   {\mathcal{O}}_{G,\mathcal{U}}   $ is a universal tight $*$-representation in the sense of \cite{exe21}, and hence ${\mathcal{O}}_{G,\mathcal{U}}\cong {C^*}_{\mathrm{tight}}({\mathcal{S}}_{G,\mathcal{U}})$.$\;$The
elements of ${\mathcal{S}}_{G,\mathcal{U}}$ come from the identification (\ref{(3.1)}) and their operations are corresponding with those for $s_{\alpha}u_{A,g} s^*_{\beta}$ in ${\mathcal{O}}_{G,\mathcal{U}}$.

In order to define an appropriate inverse semigroup, we need to add a universal path $\omega$ of length zero such that $\omega \alpha=\alpha \omega =\alpha $ for all $\alpha\in \mathcal{U}^{\geq 1}$. This is compatible with the inverse semigroup associated to a labelled graph in \cite{kum00}, where the edges are labelled with words, and $\om$ is considered as the empty word.

\begin{defn}
Let $(G,\mathcal{U},\varphi)$ be a self-similar ultragraph. The \emph{universal zero-length path} is denoted by $\om$ which satisfies $\om \om=\om$ and $\om \alpha=\alpha \om=\alpha$ for every $\alpha\in \mathcal{U}^{\geq 1}$. So, we will put $s(\om)=r(\om)=U^0$ in computations if the source or range of $\om$ is required. Moreover, we define $g \cdot \om=\om$ and $\varphi(g,\om)=g$ for every $g\in G$.
\end{defn}

\begin{defn}\label{defn4.1}
Let $(G,\mathcal U,\varphi)$ be a self-similar ultragraph and let $\mathcal{U}^\sharp:=\{\om\} \cup \mathcal{U}^{\geq 1}$. Associated to $(G,\mathcal U,\varphi)$, consider the set
\begin{equation}\label{(5...1)}
{\mathcal{S}}_{G,\mathcal{U}}: =\{(\alpha,A,g,\beta): \alpha,\beta \in \mathcal{U}^\sharp, ~ A\in {\mathcal{U}}^0, ~ g\in G ~ \text{and}\; \emptyset\neq A\subseteq{s(\alpha)\cap g \cdot s(\beta})\}\cup \{0\},
\end{equation}
and for every $(\alpha,A,g,\beta),(\gamma,B,h,\delta)$ in ${\mathcal{S}}_{G,\mathcal{U}}$ define the multiplication
\begin{align*}
(\alpha, & A,g,\beta)(\gamma,B,h,\delta)=\\
&\left\{
  \begin{array}{ll}
    \big(\alpha(g\cdot \varepsilon),(g\cdot s(\varepsilon))\cap (\varphi(g,\varepsilon)\cdot B),\varphi(g,\varepsilon)h,\delta \big) & \text{if} ~ \gamma=\beta\varepsilon, ~ \varepsilon\in \mathcal{U}^{\geq 1}, ~ g \cdot r(\varepsilon)\subseteq A, \\
    \big(\alpha, A \cap (g{\varphi(h^{-1},\varepsilon)}^{-1} h^{-1}\cdot s(\beta)),g{\varphi(h^{-1},\varepsilon)}^{-1},\delta (h^{-1}\cdot\varepsilon) \big) & \text{if} ~ \beta=\gamma\varepsilon, ~ \varepsilon\in \mathcal{U}^{\geq 1}, ~ r(\varepsilon)\subseteq B, \\
    \big(\alpha,A\cap (g\cdot  B),g h,\delta \big) & \text{if} ~ \gamma=\beta, ~ A\cap (g \cdot B) \neq \emptyset, \\
    ~ ~ 0 & \text{otherwise},
  \end{array}
\right.
\end{align*}
and the inverse
$$(\alpha,A,g,\beta)^*:=(\beta,g^{-1}\cdot A,g^{-1},\alpha).$$
\end{defn}

 Next, we show that $\mathcal{S}_{G,\mathcal{U}}$ is an inverse semigroup.

\begin{prop}\label{prop4.2}
${\mathcal{S}_{G,\mathcal{U} }}$ defined above, is an inverse semigroup with zero.
\end{prop}

\begin{proof}
First, we show that  ${\mathcal{S}_{G,\mathcal{U} }}$ is associative. So, for every $s_1=(\alpha,A,g,\beta),s_2=(\gamma,B,h,\delta)$ and $s_3=(\mu,C,f,\nu)$ in $\mathcal{S}_{G,\mathcal{U}}$, we must prove
\begin{equation}\label{(5..1)}
s_1(s_2 s_3)=(s_1 s_2)s_3.
\end{equation}
According to the multiplication in Definition \ref{defn4.1}, we must consider several cases. In the case $\gamma=\beta\varepsilon$ and $\mu=\delta\eta$ such that $\varepsilon, \eta\in \mathcal{U}^{\geq 1}$ with $ g\cdot r(\varepsilon)\subseteq A$ and $h\cdot r(\eta)\subseteq B$, Definition \ref{defn4.1} implies
\begin{align}
s_1(s_2 s_3)&=(\alpha,A,g,\beta)[(\gamma,B,h,\delta)(\mu,C,f,\nu)]\\
&=(\alpha,A,g,\beta)[(\gamma(h\cdot \eta),(h \cdot s(\eta) )\cap (\varphi(h,\eta)\cdot C),\varphi(h,\eta)f,\nu) ]\nonumber\\
&=(\alpha,A,g,\beta)[(\beta\varepsilon(h\cdot \eta),(h \cdot s(\eta) )\cap (\varphi(h,\eta)\cdot C),\varphi(h,\eta)f,\nu) ] \nonumber \\
&=\big(\alpha(g\cdot \varepsilon (h\cdot \eta)), (g\cdot s(\varepsilon (h\cdot \eta))) \cap (\varphi(g,\varepsilon(h\cdot \eta))\cdot C),\varphi(g,\varepsilon(h\cdot \eta))f,\nu \big)  \nonumber\\
&=\big(\alpha (g\cdot \varepsilon (h\cdot \eta)),(gh\cdot s(\eta)) \cap( gh\cdot C),ghf,\nu \big). \nonumber
\end{align}
(The last equality is obtained from the facts $s(\varepsilon (h \cdot \eta))=h \cdot s(\eta)$ and $\varphi(g, h \cdot \xi)\cdot C= g h \cdot C$ for each $\xi\in \mathcal{U}^{\geq 1}$.)
On the other side, we have
\begin{align}
(s_1s_2)s_3 &=[(\alpha,A,g,\beta)(\gamma,B,h,\delta)](\mu,C,f,\nu)\\
&=[(\alpha (g\cdot \varepsilon),(g\cdot s(\varepsilon))\cap (\varphi(g,\varepsilon)\cdot B),\varphi(g,\varepsilon) h ,\delta ) ](\mu,C,f,\nu)\nonumber\\
&=[(\alpha(g\cdot \varepsilon),(g\cdot s(\varepsilon))\cap (\varphi(g,\varepsilon)\cdot B),\varphi(g,\varepsilon)h,\delta)](\delta\eta,C,f,\nu)\nonumber \\
&=\big(\alpha(g\cdot \varepsilon) (\varphi(g,\varepsilon) h \cdot \eta),(\varphi(g,\varepsilon) h\cdot s(\eta))\cap(\varphi(\varphi(g,\varepsilon)h,\eta)\cdot C) ,\varphi(\varphi(g,\varepsilon) h,\eta) f ,\nu \big)  \nonumber\\
&=\big(\alpha (g\cdot (\varepsilon (h\cdot \eta))),(gh\cdot s(\eta) )\cap (gh\cdot C),ghf,\nu \big),\nonumber
\end{align}
concluding $s_1(s_2 s_3)=(s_1 s_2)s_3$ in this case. In the other cases, (\ref{(5..1)})  can be verified analogously,  hence ${\mathcal{S}_{G,\mathcal{U} }}$ is associative.

Moreover, for any $s=(\alpha,A,g,\beta)$ in ${\mathcal{S}_{G,\mathcal{U} }}$, we get
\begin{align*}
ss^*s&=(\alpha,A,g,\beta)(\beta,g^{-1}\cdot A,g^{-1},\alpha)(\alpha,A,g,\beta)\\
&=(\alpha,A,{{1}_G},\alpha)(\alpha,A,g,\beta)\\
&=(\alpha,A,g,\beta) \\
&=s,
\end{align*}
and similarly $s^*ss^*=s^*$. So it remains to prove that $s^*$ is the unique element of ${\mathcal{S}_{G,\mathcal{U} }}$ satisfying the above identity. By \cite[Theorem 1.1.3]{exe17}, it is equivalent to show that $ss^* $ commutes with  $tt^*$ for all $s,t \in {\mathcal{S}_{G,\mathcal{U} }}$. In order to do this, fix arbitrary elements $s=(\alpha,A,g,\beta)$ and $t=(\gamma,B,h,\delta)$ in ${\mathcal{S}}_{G,\mathcal{U}}$. In the case $\gamma=\alpha\varepsilon$ with $\varepsilon\in \mathcal{U}^{\geq 1}$ and $ r(\varepsilon)\subseteq A$, we have
\begin{align*}
(ss^*)(tt^*)&=[(\alpha,A,g,\beta)(\beta,g^{-1}\cdot A,g^{-1},\alpha)][(\gamma,B,h,\delta)(\delta,h^{-1}\cdot B,h^{-1},\gamma)]\\
&=(\alpha,A,{{1}_G},\alpha)(\gamma,B,{{1}_G},\gamma)\\
&=(\alpha\varepsilon,s(\varepsilon)\cap B,{{1}_G},\gamma)\\
&=(\gamma,s(\gamma)\cap B,{{1}_G},\gamma)\\
&=(\gamma,B,{{1}_G},\gamma),
\end{align*}
and also
\begin{align*}
(tt^*)(ss^*)&=[(\gamma,B,h,\delta)(\delta,h^{-1}\cdot B,h^{-1},\gamma)][(\alpha,A,g,\beta)(\beta,g^{-1}\cdot A,g^{-1},\alpha)]\\
&=(\gamma,B,{{1}_G},\gamma)(\alpha,A,{{1}_G},\alpha)\\
&=(\gamma,s(\gamma)\cap B,{{1}_G},\alpha\varepsilon) \\
&=(\gamma,s(\gamma)\cap B,{{1}_G},\gamma),\\
&=(\gamma,B,{{1}_G},\gamma).
\end{align*}
So, we obtain $(ss^*)(tt^*)=(tt^*)(ss^*)$ in this case. Again, checking this equality in the other cases is straightforward and left to the reader. Therefore, ${\mathcal{S}_{G,\mathcal{U} }}$ is an inverse semigroup.
\end{proof}

The proof of Proposition \ref{prop4.2} shows that every element $ss^*$ is of the form $(\alpha,A,{{1}_G},\alpha)$. So the idempotent semilattice of ${\mathcal{S}_{G,\mathcal{U} }}$ is
$$\mathcal{E}({\mathcal{S}_{G,\mathcal{U}}})=\{(\alpha, A, {{1}_G},\alpha): \alpha \in \mathcal{U}^\sharp, A\subseteq s(\alpha)\}.$$

\begin{defn}\label{defn5.4}
The \emph{ultrapath space} of $\mathcal{U}$ is denoted by
$$\mathcal{P}=\{(\alpha,A)\in \mathcal{U}^\sharp \times \mathcal{U}^0: \emptyset \neq A \subseteq s(\alpha)\}.$$
\end{defn}

Therefore, $(\alpha,A)\mapsto (\alpha,A,{{1}_G},\alpha)$ is a one-to-one correspondence between the ultrapath space of $\mathcal{U}$ and idempotent semilattice $\mathcal{E}({\mathcal{S}_{G,\mathcal{U}}})$. In the sequel, to ease the notation, we denote the idempotent
$$q_{(\alpha,A)}:=(\alpha,A,{{1}_G},\alpha)$$
in ${\mathcal{S}_{G,\mathcal{U} }}$ for any ultrapath $(\alpha,A) \in \mathcal{P}$.

\begin{prop}
Let $(G,\mathcal{U},\varphi)$ be a self-similar ultragraph. For every $q_{(\alpha,A)},q_{(\beta,B)}\in\mathcal{E}({\mathcal{S}_{G,\mathcal{U} }}) $, we have
$$q_{(\alpha,A)}q_{(\beta,B)}=
\left\{
  \begin{array}{ll}
    q_{(\alpha,A\cap B)} & \alpha=\beta ~ \text{and} ~ A\cap B\neq \emptyset, \\
    q_{(\alpha,A)} & \alpha=\beta\gamma ~ with ~ \gamma\in \mathcal{U}^{\geq 1} ~ and ~ r(\gamma)\subseteq B, \\
    q_{(\beta,B)} & \alpha\gamma=\beta ~ with ~ \gamma\in \mathcal{U}^{\geq 1} ~ and ~  r(\gamma)\subseteq A, \\
    0 & \text{otherwise}.
  \end{array}
\right.
$$
\end{prop}

\begin{proof}
If $\alpha=\beta$ and $A\cap B\neq \emptyset$, then Definition \ref{defn4.1} says $(\alpha,A,1_G,\alpha)(\alpha,B,1_G,\alpha)=(\alpha,A\cap B,1_G,\alpha)$. Also, if $\alpha=\beta\gamma$ and $r(\gamma)\subseteq B$ then
$$(\alpha,A,{{1}_G},\alpha)(\beta,B,{{1}_G},\beta)=(\alpha,A \cap s(\alpha),{{1}_G},\beta\gamma)=(\alpha,A,{{1}_G},\alpha),$$
while in the case $\beta=\alpha\gamma$ with $r(\gamma)\subseteq A$, we have
$$(\alpha,A,{{1}_G},\alpha)(\beta,B,{{1}_G},\beta)=(\alpha\gamma,s(\beta) \cap B,{{1}_G},\beta)=(\beta,B,{{1}_G},\beta).$$
In the other cases, Definition \ref{defn4.1} implies $q_{(\alpha,A)}q_{(\beta,B)}=0.$
\end{proof}

The above proposition concludes the following result.

\begin{cor}\label{cor4.4}
Let $(G,\mathcal{U},\varphi)$ be a self-similar ultragraph and $q_{(\alpha,A)},q_{(\beta,B)}\in\mathcal{E}({\mathcal{S}_{G,\mathcal{U} }}) $. Then $q_{(\alpha,A)}\le q_{(\beta,B)}$ if and only if one of the following holds:
\begin{enumerate}[(a)]
 \item $\alpha=\beta$ and $A\subseteq B$, or
 \item $\alpha=\beta\gamma$ for some $\gamma \in {\mathcal{U}}^{\geq 1}$ with $r(\gamma)\subseteq B$.
\end{enumerate}
\end{cor}


\section{Tight spectrum and tight groupoid}\label{sec80}

Given an ultragraph ${\mathcal{U} }$, it is introduced in \cite{bed17} the inverse semigroup
$$\mathcal{S}_{\mathcal{U} }=\{ (\alpha,A,\beta) :  \alpha,\beta \in \mathcal{U}^\sharp, A\in {\mathcal{U} }^0\}$$
for $\mathcal{U}$ such that $C^*(\mathcal{U})\cong C^*_{\mathrm{tight}}({\mathcal{S}_{\mathcal{U} }})\cong C^*({\mathcal{G }_{\mathrm{tight}}}({\mathcal{S}_{\mathcal{U} }}))$. By $(\alpha,A,\beta)\longmapsto (\alpha,A,1_G,\beta)$, $\mathcal{S}_{\mathcal{U} }$ can be embedded in our inverse semigroup ${\mathcal{S}}_{G,\mathcal{U}}$ of Definition \ref{defn4.1} with the same idempotent set
$$\mathcal{E}({\mathcal{S}}_{G,\mathcal{U}})=\{q_{(\alpha,A)}:\alpha\in \mathcal{U}^\sharp, \emptyset \neq A\subseteq s(\alpha)\}$$
So, their filter spaces are same and we may apply the description of tight filters in \cite[Proposition 3.6]{bed17} for tight filters of $\mathcal{S}_{G,\mathcal{U}}$. In this section, we briefly recall the tight spectrum ${\widehat{\mathcal{E}}}_{\mathrm{tight}}({\mathcal{S}}_{G,\mathcal{U}})$ from \cite{bed17} and review the construction of tight groupoid ${\mathcal{G }_{\mathrm{tight}}}({\mathcal{S}_{G,\mathcal{U} }})$.$\;$Then, in the next section, we will describe the action of $\mathcal{S}_{G,\mathcal{U} }$ on ${\widehat{\mathcal{E}}}_{\mathrm{tight}}({\mathcal{S}}_{G,\mathcal{U}})$ in details.

Let $\mathcal{U}=(U^0,\mathcal{U}^1,r,s)$ be an ultragraph. We may consider $\mathcal{U}^0$ as a partially ordered (semi)lattice with the meet $A \wedge B:=A \cap B$ and the partial order $A\subseteq B$. Similarly, for every subset $X\subseteq U^0$, the family
$${\mathcal{B}}_X=\{ A \in \mathcal{U}^0:A \subseteq X  \}$$
is a partially order semilattice as well.

\begin{prop}[{\cite{bed17}}]\label{prop6.1}
Let $(G,\mathcal{U},\varphi)$ be a self-similar ultragraph and ${\mathcal{S}}_{G,\mathcal{U}}$ the associated inverse semigroup. Then every tight filter $\mathcal{F}$ in ${\widehat{\mathcal{E}}}_{\mathrm{tight}}({\mathcal{S}}_{G,\mathcal{U}})$ can be (uniquely) described as one of the following forms:
\begin{enumerate}
  \item $\mathcal{F}$ is associated to a pair $(\alpha,\mathcal{B})$, where $\alpha\in \mathcal{U}^\sharp$ and $\mathcal{B}$ is a filter in the set
$${\mathcal{B}}_{s(\alpha)}:=\{A \in {{\mathcal{U} }^0} : A\subseteq s(\alpha)\},$$
such that $|A|=\infty$ for all $A \in \mathcal{B}$. Then
$$\mathcal{F}={\mathcal{F}}_{(\alpha,\mathcal{B})}:=\{q_{(\alpha,A)}: A \in \mathcal{B}\} \cup \{q_{(\beta,A)} : |\beta|<|\alpha|, \beta={\alpha}_1\ldots \alpha_{|\beta|} ~ \text{and} ~ r(\alpha_{{|\beta|}+1})\subseteq A\}.$$
  \item $\mathcal{F}$ is associated to an infinite path $x=\alpha_1\alpha_2 \cdots \in {\mathcal{U}}^{\infty}$ such that
$$\mathcal{F}={\mathcal{F}}_x:=\{q_{(\beta,A)}\;:\: \beta={\alpha}_1\alpha_2\cdots \alpha_{|\beta|}\; \text{and} \;r(\alpha_{{|\beta|}+1})\subseteq A \subseteq s(\alpha_{|\beta|})\}.$$
\end{enumerate}
Consequently, the tight spectrum $\widehat{\mathcal{E}}_{\mathrm{tight}}({\mathcal{S}}_{G,\mathcal{U}})$ is precisely the set of all filters ${\mathcal{F}}_{(\alpha,\mathcal{B})}$ and ${\mathcal{F}}_x$ which are defined above.
\end{prop}

Note that, in the tight filters ${\mathcal{F}}_{(\alpha,\mathcal{B})}$, if $\alpha=\om$ then $\mathcal{B}$ is a filter in $\mathcal{U}^0$ and we have
$${\mathcal{F}}_{(\om,\mathcal{B})}=\{q_{(\om,A)}: A\in \mathcal{B}\}.$$

{\bf{ Notation}.}
In the sequel, the tight spectrum ${\widehat{\mathcal{E}}}_{\mathrm{tight}}({\mathcal{S}}_{G,\mathcal{U}})$ characterized above will be denoted by $\mathcal{T}$ for notational convenience.

We recall a basis of compact open sets for the pointwise convergent topology on $\mathcal{T}\subseteq \widehat{\mathcal{E}}_{0}({\mathcal{S}}_{G,\mathcal{U}})$. For any finitely many idempotents $e,{e}_1 ,\ldots ,{e}_n \in {{\mathcal{E}}({\mathcal{S}_{G,{\mathcal{U}}}}})$, define
$$V_{e : {e}_1 ,\ldots ,{e}_n}=V_{e}\backslash \bigcup_{i=1}^{n}{V_{e_i}}=\{{{\mathcal{F}}}\in\mathcal{T} : e\in {{\mathcal{F}}}, \; \;{e}_1, \ldots ,{e}_n \notin {{\mathcal{F}}} \},$$
where $$V_{e}=\{{\mathcal{F}} \in\mathcal{T}  :e\in {{\mathcal{F}}}\}.$$
Then {\cite[Lemmas 2.22 and 2.23]{exe16}} imply that the collection $\{V_{e : {e}_1 ,\ldots ,{e}_n}\;:\;n\geq 0,\; e,e_i \in {{\mathcal{E}}({\mathcal{S}_{G,{\mathcal{U}}}}})\}$ is a basis of compact open sets for the inherited topology on $\mathcal{T}$. Moreover, this topology is Hausdorff. Note that for a filter ${\mathcal{F}}$ we have
$$ e,e'\in {\mathcal{F}}\quad \Longleftrightarrow \quad \quad ee'\in {\mathcal{F}},$$
hence $V_e \cap V_{e'}=V_{ee'}$. This turns out $V_{e : {e}_1 ,\ldots ,{e}_n}=V_e \backslash (V_{e e_1 }\cup\ldots\cup V_{e e_n })$, so in each set $V_{e : {e}_1, \ldots ,{e}_n}$, we can assume ${e}_1, \ldots, {e}_n \leq e$.

\begin{defn}
For any ultrapath $(\alpha,A)$ in $\mathcal{U}$ define
$$\mathrm{Z}(\alpha,A):=V_{q_{(\alpha,A)}}=\{{\mathcal{F}} \in \mathcal{T}: q_{(\alpha,A)}\in {\mathcal{F}}\} $$
which is the set of all tight filters in $\mathcal{T}$ containing $q_{(\alpha,A)}$.
\end{defn}

Furthermore, if ${{\mathcal{F}}}$ is an ultrafilter in $\mathcal{T}$, then \cite[Proposition 2.5]{exe16} says that the sets  $\{\mathrm{Z}(\alpha,A): q_{(\alpha,A)}\in \mathcal{F}\}$ form a neighborhood basis for ${\mathcal{F}}$ in ${\widehat{\mathcal{E}}}_{\infty}({\mathcal{S}}_{G,\mathcal{U}})$. Since $\mathcal{T}={\widehat{\mathcal{E}}}_{\infty}({\mathcal{S}}_{G,\mathcal{U}})$ by \cite[Corollary 6.2]{kum00}, we conclude:

\begin{prop}\label{prop6.3}
The sets $\mathrm{Z}(\alpha,A)$ are a basis of compact open sets for the topology on $\mathcal{T}$.
\end{prop}

As said before, the tight groupoid of germs for ${\mathcal{S}_{G,\mathcal{U} }}$ is
$${\mathcal{G }_{\mathrm{tight}}}({\mathcal{S}_{G,\mathcal{U} }})=\{[s,\mathcal{F}] :( s,\mathcal{F})\in {\mathcal{S}_{G,\mathcal{U} }}*\mathcal{T}  ,\;s^*s\in \mathcal{F}\}$$
(see Subsection \ref{sub2.3}). So, the topology on $\mathcal{T}$ induces a natural topology on ${\mathcal{G}}_{\mathrm{tight}}(\mathcal{S}_{G,{\mathcal{U} }})$ consisting the sets of the form $[s,V]=\{[s,\mathcal{F}]:\; \;\mathcal{F} \in V\}$ for open sets $V$ in $\mathcal{T}$ and $s \in \mathcal{S}_{G,{\mathcal{U} }}$. By applying Proposition \ref{prop6.3}, we may consider the sets $[(\alpha,A,g,\beta),\mathrm{Z}(\gamma,B)]$. Let $(\alpha,A,g,\beta)\in {\mathcal{S}_{G,{\mathcal{U} }}}$ and $(\gamma,B)$ be an ultrapath in $\mathcal{U}$. Clearly $s^*s=(\beta,g^{-1}\cdot A,{{1}_G},\beta)\in \mathrm{Z}(\gamma,B)$ implies that $\beta$ is a prefix of $\gamma$, where in the case $\gamma=\beta$  we have also $B\subseteq g^{-1}\cdot A$.

Moreover, in the case $\gamma=\beta \gamma'$ with $\gamma'\in \mathcal{U}^{\geq 1}$ such that $g\cdot r(\gamma')\subseteq A$, for each $[(\alpha,A,g,\beta),\mathcal{F}]\in [(\alpha,A,g,\beta),\mathrm{Z}(\gamma,B)]$, the identity
$$(\alpha,A,g,\beta)(\gamma,B,{{1}_G},\gamma)=\big(\alpha(g\cdot \gamma'), (g \cdot s(\gamma'))\cap (\varphi(g,\gamma')\cdot B), \varphi(g,\gamma'),\gamma \big)$$
gives
$$[(\alpha,A,g,\beta), \mathcal{F}]=[\big(\alpha(g\cdot \gamma'), (g \cdot s(\gamma'))\cap (\varphi(g,\gamma')\cdot B), \varphi(g,\gamma'),\gamma         \big),\mathcal{F}]$$
for all $\mathcal{F} \in \mathrm{Z}(\gamma,B)$. Therefore
$$[(\alpha,A,g,\beta),\mathrm{Z}(\gamma,B)]=[\big(\alpha(g\cdot \gamma'), (g \cdot s(\gamma'))\cap (\varphi(g,\gamma')\cdot B), \varphi(g,\gamma'), \gamma         \big),\mathrm{Z}(\gamma,B)],$$
and we can only consider the sets of the form $[(\alpha,A,g,\beta),\mathrm{Z}(\beta,B)].$
Furthermore, by putting $ \gamma=\beta$ in the above computation, it follows
$$[(\alpha,A,g,\beta), \mathcal{F}]=[(\alpha, A\cap (g\cdot B),g,\beta),\mathcal{F}]\qquad (\forall ~~ \mathcal{F} \in \mathrm{Z}(\beta,B)),$$
and consequently,
$$[(\alpha,A,g,\beta),\mathrm{Z}(\beta,B)]=[(\alpha,A \cap (g \cdot B),g,\beta),\mathrm{Z}(\beta,(g^{-1}\cdot A) \cap B)].$$
Therefore, we obtain the following proposition.

\begin{prop}\label{prop8004}
Let $(G,\mathcal{U},\varphi)$ be a self-similar ultragraph. Then ${\mathcal{G}}_{\mathrm{tight}}(\mathcal{S}_{G,{\mathcal{U} }})$ is an ample groupoid with Hausdorff unit space. Moreover, the sets
$$\Theta(\alpha,A,g,\beta):=[(\alpha,A,g,\beta), \mathrm{Z}(\beta,g^{-1}\cdot A)]$$
for $\alpha,\beta \in \mathcal{U}^\sharp$, $g \in G$ and $A\subseteq s(\alpha)\cap g\cdot s(\beta)$, are compact open bisections generating the topology on ${\mathcal{G}}_{\mathrm{tight}}(\mathcal{S}_{G,{\mathcal{U} }})$.
\end{prop}


\section{The action of $\mathcal{S}_{G,{\mathcal{U} }}$ on $\mathcal{T} $ }\label{sec800}

After describing the tight spectrum $\mathcal{T}={\mathcal{G}}_{\mathrm{tight}}(\mathcal{S}_{G,{\mathcal{U} }})$ in the previous section, we are going to compute the action $\theta:\mathcal{S}_{G,{\mathcal{U} }}     \curvearrowright   \mathcal{T}$. Recall that, for any $s\in \mathcal{S}_{G,{\mathcal{U} }}$, the map $\theta_s :D^{s^*s}\rightarrow D^{ss^*}$ is defined by $\theta_s (\mathcal{F})=(s\mathcal{F}s^*)\uparrow$. Also if $s=(\alpha,A,g,\beta)$, then $s^*s=(\beta,g^{-1}\cdot A ,1_G,\beta)$ and we have
$$D^{s^*s}=\{\mathcal{F}\in \mathcal{T}\;:\; s^*s\in \mathcal{F}\}=\mathrm{Z}(\beta,g^{-1}\cdot A),$$
where each $\mathcal{F}\in \mathcal{T}$ is of the forms $\mathcal{F}_{(\gamma,{\mathcal{B}})}$ and $\mathcal{F}_{x}$
being described in Proposition \ref{prop6.1}. In order to compute $\theta_{s }({\mathcal{F}})$ precisely, we should divide our proofs for $s=(\alpha,A,g,\beta)$ to the cases $\beta=\om$ and $\beta\in {\mathcal{U} }^{\geq1}$. In the first case, given a tight filter of the form $\mathcal{F}_{(\om,\mathcal{B})}\in \mathcal{T}$, where $\mathcal{B}$ is a filter in $\mathcal{U}^0$, we may have $\theta_{s}({\mathcal{F}}_{(\om,\mathcal{B})})={\mathcal{F}}_{(\alpha,{\mathcal{B}}')}$ for a suitable filter ${\mathcal{B}}'$ in $\mathcal{B}_{s(\alpha)}$. Indeed, we will see in the next proposition that ${\mathcal{B}}'$ would be
$$g\cdot \mathcal{B}\downarrow_{s(\alpha)}:=\{(g\cdot B)\cap s(\alpha): B\in  {\mathcal{B}}  \}.$$
Recall that in case $\alpha=\om$, then $g\cdot \mathcal{B}\downarrow_{s(\alpha)}$ is equal to $g \cdot \mathcal{B}=\{g \cdot B: B\in \mathcal{B}\}$.

First a lemma:

\begin{lem}
 Let $(G,\mathcal{U},\varphi)$ be a self-similar ultragraph and fix $g\in G$. If $\mathcal{B}$ is an ultrafilter in ${\mathcal{U} }^0$, then so is $g\cdot \mathcal{B}$.
\end{lem}

\begin{proof}
This follows from the fact that the map $A \mapsto g\cdot A$ is an automorphism preserving the set operation $\cap$ and the partial order $\subseteq$.
\end{proof}

\begin{prop}\label{prop11.3}
Let $s=(\alpha,A,g,\om)$ be an element of $\mathcal{S}_{G,{\mathcal{U}}}$. Then

\begin{enumerate}[(1)]
    \item For each $\mathcal{F}_{(\om,\mathcal{B})}\in \mathrm{Z} (\om,g^{-1}\cdot A)$, we have
$$\theta_{s}(\mathcal{F}_{(\om,\mathcal{B})})={\mathcal{F}}_ {(\alpha,g\cdot \mathcal{B}\downarrow_{s(\alpha)})}.$$
    \item For each ${\mathcal{F}}_ {(\gamma,{\mathcal{B}})}\in \mathrm{Z} (\om,g^{-1}\cdot A)$ with $\gamma\in \mathcal{U}^{\geq 1}$, we have
$$\theta_{s}({\mathcal{F}}_ {(\gamma,{\mathcal{B}})})={\mathcal{F}}_{(\alpha(g\cdot\gamma),\varphi(g,\gamma )\cdot  \mathcal{B})}.$$
\end{enumerate}
\end{prop}

\begin{proof}
(1). Let $\mathcal{F}_{(\om,\mathcal{B})}\in \mathrm{Z} (\om,g^{-1}\cdot A)$. By definition, $q_{(\om,g^{-1}\cdot A)}\in \mathcal{F}_{(\om,\mathcal{B})}$, so $g^{-1}\cdot A \in {\mathcal{B}}$ or equivalently $A\in g\cdot {\mathcal{B}}$. In particular, we have $\emptyset    \neq     A\cap (g\cdot B)\in g\cdot {\mathcal{B}}$ for all $B\in {\mathcal{B}}$ because $g\cdot {\mathcal{B}}$ is a filter in ${\mathcal{U}}^0$.

In the case $\alpha\in \mathcal{U}^{\geq 1}$, for any $q_{(\om,B)}\in \mathcal{F}_{(\om,\mathcal{B})}$, we may compute
\begin{align*}
sq_{(\om,B)} s^* &=(\alpha,A,g,\om)(\om,B,1_G,\om)(\om,g^{-1}\cdot A,g^{-1},\alpha)\\
&=(\alpha,A\cap(g\cdot B ),1_G,\alpha) \\
&=q_{(\alpha,A\cap(g\cdot B ))}.
\end{align*}
Hence, the definition of $\theta_s :D^{s^*s}\rightarrow D^{ss^*}$ says that
\begin{align*}
\theta_{s}(\mathcal{F}_{(\om,\mathcal{B})})
&=(s\mathcal{F}_{(\om,\mathcal{B})} s^* )\uparrow\\
&=\{ sq_{(\om,B)}s^*: B\in {{\mathcal{B}}}\}\uparrow\\
&=\{q_{(\alpha,A\cap (g\cdot B))} :  B\in {{\mathcal{B}}}\}\uparrow \\
&=\{q_{(\alpha,g\cdot B)} :  B\in {{\mathcal{B}}} \text{ and}\; B\subseteq g^{-1} \cdot A\}\uparrow  \qquad \text{(because}\; A\cap (g\cdot B) \in g\cdot {{\mathcal{B}}}) \\
&={\mathcal{F}}_ {(\alpha,g\cdot \mathcal{B}\downarrow_{s(\alpha)})}.
\end{align*}

In the case $\alpha =\om$, we can analogously obtain $\theta_{s}(\mathcal{F}_{(\om,\mathcal{B})})={\mathcal{F}}_ {(\om,g\cdot {\mathcal{B}})}.$

For statement $(2)$, let ${\mathcal{F}}_ {(\gamma,{\mathcal{B}})}$ lie in $\mathrm{Z} (\om,g^{-1}\cdot A)$ with $|\gamma|\geq 1$. Similar to part (1), if $q_{(\beta,B)}\in {\mathcal{F}}_ {(\gamma,{\mathcal{B}})}$ (refer to Proposition \ref{prop6.1} for the definition of ${\mathcal{F}}_ {(\gamma,{\mathcal{B}})}$), then
\begin{align}
sq_{(\beta,B)} s^*
&=(\alpha,A,g,\om)(\beta,B,1_G,\beta)(\om,g^{-1}\cdot A,g^{-1},\alpha) \nonumber\\
&=(\alpha(g\cdot \beta),\varphi(g,\beta) \cdot B,1_G,\alpha(g\cdot \beta)) \nonumber\\
&=q_{(\alpha(g\cdot \beta),\varphi(g,\beta) \cdot B )}.\label{(11.1)}
\end{align}

Note that $q_{(\beta,B)}\in {\mathcal{F}}_ {(\gamma,{\mathcal{B}})}$ implies that $\beta$ is a prefix for $\gamma$, while in the case $\beta=\gamma$ we must have $B\in {\mathcal{B}}$. Therefore, $ s q_{(\beta,B)} s^*\in {\mathcal{F}}_{(\alpha(g\cdot \gamma), \varphi(g,\gamma)\cdot B  )}$ by (\ref{(11.1)}) above, and we conclude that
\begin{align*}
\theta_{s}({\mathcal{F}}_ {(\gamma,{\mathcal{B}})})
&=(s({\mathcal{F}}_ {(\gamma,{\mathcal{B}})})s^*)\uparrow \\
&=\{q_{(\alpha(g\cdot \beta),\varphi(g,\beta) \cdot B )} : q_{(\beta,B)}\in {\mathcal{F}}_ {(\gamma,{\mathcal{B}})}\}\uparrow \\
&={\mathcal{F}}_ {( \alpha(g\cdot \gamma),\varphi(g,\gamma)\cdot  {\mathcal{B}})},
\end{align*}
completing the proof.
\end{proof}

Next, we compute $\theta_{s}({\mathcal{F}}_ {x})$ for $x\in {\mathcal{U}}^\infty$. Since $D^{s^*s}= \mathrm{Z} (\beta,g^{-1}\cdot A)$ for $s=(\alpha,A,g,\beta)$, $x$ must be of the form $x=\beta y$ with $r(y)\subseteq g^{-1}\cdot A$.

\begin{prop}
Let $s=(\alpha,A,g,\beta)$ be in $\mathcal{S}_{G,{\mathcal{U} }}$. For every ${\mathcal{F}}_{\beta x}\in \mathrm{Z} (\beta,g^{-1}\cdot A)$ with $x\in {\mathcal{U}}^\infty$, we have
$$\theta_{s}({\mathcal{F}}_ {\beta x})={\mathcal{F}}_ {\alpha(g\cdot x)}.$$
\end{prop}

\begin{proof}
If ${\mathcal{F}}_{\beta x}\in \mathrm{Z} (\beta,g^{-1}\cdot A)$, then for every $q_{(\beta\gamma,B)}\in {\mathcal{F}}_ {\beta x}$ we have
\begin{align}
sq_{(\beta\gamma,B)} s^*
&=(\alpha,A,g,\beta)(\beta\gamma,B,1_G,\beta\gamma)(\beta,g^{-1}\cdot A,g^{-1},\alpha) \nonumber\\
&=(\alpha(g\cdot \gamma),\varphi(g,\gamma) \cdot B,1_G,\alpha(g\cdot \gamma))\nonumber \\
&=q_{(\alpha(g\cdot \gamma),\varphi(g,\gamma) \cdot B )}.\nonumber
\end{align}
So, the description in Proposition \ref{prop6.1}(2) follows that
$${\theta}_s ({\mathcal{F}}_ {\beta x})=(s  {\mathcal{F}}_ {\beta x} s^*)\uparrow={\mathcal{F}}_ {\alpha(g\cdot x)}$$
as desired.
\end{proof}

To prove Proposition \ref{prop7...5} below we need the following lemma.

\begin{lem}
Let $(G,\mathcal{U},\varphi)$ be a self-similar ultragraph and let $\alpha \in {\mathcal{U}}^{\geq 1}$. If $\mathcal{B}$ is an ultrafilter in ${\mathcal{B}}_{s(\alpha)}=\{A\in {\mathcal{U} }^0\;:\; A\subseteq {s(\alpha)}\}$, then
$${\mathcal{B}}\uparrow_{{\mathcal{U} }^0}:=\{ A\in{\mathcal{U} }^0 : \exists B\in {\mathcal{B}}\text{ such that }B\subseteq A\}$$
is an ultrafilter in ${\mathcal{U} }^0$.
\end{lem}

\begin{proof}
Clearly, ${{\mathcal{B}}\uparrow}_{{\mathcal{U} }^0}$ is closed under the intersection as ${\mathcal{B}}$ is. Moreover, if $A\in {{\mathcal{B}}\uparrow}_{{\mathcal{U} }^0}$ and $A \subseteq C \in {\mathcal{U} }^0$, then there is $B\in {\mathcal{B}}$ such that $B \subseteq A \subseteq C$, so $C\in {\mathcal{B}}\uparrow_{{\mathcal{U} }^0}$. This says that ${\mathcal{B}}\uparrow_{{\mathcal{U} }^0}$ is a filter in ${\mathcal{U} }^0$.

Now we show that ${\mathcal{B}}\uparrow_{{\mathcal{U} }^0}$ is an ultrafilter in ${\mathcal{U} }^0$. Assume by way of contradiction that $\mathcal{F}$ is a filter in ${\mathcal{U} }^0$ satisfying ${\mathcal{B}}\uparrow_{{\mathcal{U} }^0}\subsetneqq \mathcal{F}$. Pick some $A\in \mathcal{F}\setminus {\mathcal{B}}\uparrow_{{\mathcal{U} }^0}$. Note that for every $B \in \mathcal{B} \subseteq \mathcal{F}$ we have $A\cap B \neq \emptyset$ because $\mathcal{F}$ is a filter. Define
$${\mathcal{B}}':={\mathcal{B}} \cup \{A\cap B\; :\; B\in {\mathcal{B}}\}$$
which is a subset of ${\mathcal{B}}_{s(\alpha)}$. We claim that ${\mathcal{B}}'$ is a filter in ${\mathcal{B}}_{s(\alpha)}$ containing ${\mathcal{B}}$.

To prove the claim, given $A\cap B_1, A\cap B_2\in {\mathcal{B}}'$ we have
$$(A\cap B_1)\cap(A\cap B_2)=A\cap(B_1 \cap B_2),$$
thus $(A\cap B_1)\cap(A\cap B_2)\in {\mathcal{B}}'$ because $B_1 \cap B_2 \in {\mathcal{B}}$ . Furthermore, for every $B\in {\mathcal{B}}'$ and $C\in{\mathcal{B}}_{s(\alpha)} $ with $B\subseteq C$, we show that $C \in {\mathcal{B}}'$. In case $B\in {\mathcal{B}}$, then $C\in {\mathcal{B}}$ because ${\mathcal{B}}$ is a filter in ${\mathcal{B}}_{s(\alpha)}$. In the other case, if $B$ is of the form $B=A\cap B'$ for some $B' \in{\mathcal{B}} $, since $B \subseteq C$, $C$ must be of the form $C=A\cap C' $ for some $B'\subseteq C'\in\mathcal{B} $.
Therefore, $C\in {\mathcal{B}}'$ and the claim is proved.

 Note that for each $B\in {\mathcal{B}}$, $A\cap B$ does not belong to ${\mathcal{B}}$ because $A\cap B \in {\mathcal{B}}$ implies $A\in {\mathcal{B}}$. Thus ${\mathcal{B}}'$ is a filter in ${\mathcal{B}}_{s(\alpha)} $ which properly contains ${\mathcal{B}} $, which contradicts the maximality of ${\mathcal{B}}$ in ${\mathcal{B}}_{s(\alpha)}$. Consequently, ${\mathcal{B}}\uparrow_{{\mathcal{U} }^0}$ is an ultrafilter in ${{\mathcal{U} }^0}$.
\end{proof}

The proof of next proposition is analogous to that of Proposition $\ref{prop11.3}$, so the details are left to the reader.

\begin{prop}\label{prop7...5}
Let $s=(\alpha,A,g,\beta)$ be an element of $\mathcal{S}_{G,{\mathcal{U} }}$ with $|\beta |\geq 1 $. Then, for every tight filter ${\mathcal{F}}_ {(\beta \gamma,\mathcal{B})}\in \mathrm{Z}(\beta, g^{-1} \cdot A)$, we have
$$\theta_{s}({\mathcal{F}}_ {(\beta \gamma,\mathcal{B})})=
\left\{
  \begin{array}{ll}
    {\mathcal{F}}_{\left(\om, (g\cdot \mathcal B) \uparrow_{\mathcal {U}^0}\right)} & \text{ if} ~ \alpha=\gamma=\om \\
    {\mathcal{F}}_ {\left(\alpha(g \cdot \gamma), \varphi(g,\gamma) \cdot \mathcal{B}\right)} & \text{ if} ~|\alpha|\geq 1 ~ \text{or} ~ |\gamma|\geq 1.
  \end{array}
\right.
$$
\end{prop}


\section{The tight representation}\label{sec8}

The tight $C^*$-algebra of an inverse semigroup $\mathcal{S}$ with $0$, denoted by $C^*_{\mathrm{tight}}(\mathcal{S})$, is introduced in \cite{{exe21}}, which is generated by a universal tight representation of $\mathcal{S}$. In this section, for a self-similar ultragraph $(G,\mathcal{U},\varphi)$ we prove that $C^*_{\mathrm{tight}}({\mathcal{S}_{G,\mathcal{U} }})\cong {\mathcal{O}_{G,\mathcal{U} }}$. Recall from \cite[Definition 13.1]{{exe21}} that a representation $\pi:\mathcal{S}\rightarrow B(H)$ on a Hilbert space $H$ is called \emph{tight} if for every $X,Y\subseteq \mathcal{E}(\mathcal{S})$ and every finite cover $Z$ for the set
$$\mathcal{E}(\mathcal{S})^{X,Y}:=\left\{e\in \mathcal{E}(\mathcal{S}): e\leq f, \;\; \mathrm{for \;all}\;\; f\in X, ~ \mathrm{and} ~ ef'=0 \;\; \mathrm{for \;all}\;\; f'\in Y\right\},$$
one has
$$\bigvee_{e\in Z} \pi(e)=\bigwedge_{f\in X} \pi(f) \wedge \bigwedge_{f'\in Y}(1-\pi(f')).$$
In addition, we say $\pi$ is a \emph{universal tight representation} if for every tight representation $\pi:\mathcal{S}\rightarrow \mathcal{A}$, there exists a $*$-homomorphism $\psi: \pi(\mathcal{S})\rightarrow \mathcal{A}$ such that $\psi \circ \pi=\phi$.

\begin{defn}[\cite{exe16}]
Let $\mathcal{S}$ be an inverse semigroup with zero and $s$ a fixed element in $\mathcal{S}$.
\begin{enumerate}
  \item The {\emph{principal ideal}} of $\mathcal{E}(\mathcal{S})$ generated by $s$ is defined by
$$\mathcal{J}_s:=\{e\in \mathcal{E}(\mathcal{S}): e\leq s\},$$
that is equal to $\{e\in \mathcal{E}({\mathcal{S}_{G,\mathcal{U} }}): se=e\}$. (Recall that, for any $e\in \mathcal{E}(\mathcal{S})$, $e\leq s$ if and only if $se=e$.)
  \item A subset $\mathcal{C}\subseteq \mathcal{E}(\mathcal{S})$ is called an {\emph{outer cover}} for $s$  if
for any $e\in \mathcal{J}_s$, there exists $f \in\mathcal{C}$ satisfying $ef\neq 0$ (in this case we write $e\Cap f)$.

  \item A {\it{cover}} for $s$ is an outer cover $\mathcal{C}$ contained in $\mathcal{J}_s$.
\end{enumerate}
\end{defn}

\begin{thm}\label{thm7.2}
Let $(G,\mathcal{U},\varphi)$ be a self-similar ultragraph as in Definition $\ref{defn3.2}$. Then the natural map $\pi : \mathcal{S}_{G,{\mathcal{U}}}\rightarrow \mathcal{O}_{G,{\mathcal{U}}}$ defined by $\pi(0)=0$ and
$$\pi(\alpha,A,g,\beta)=s_{\alpha}u_{A,g}s^*_{\beta} \qquad ((\alpha,A,g,\beta)\in \mathcal{S}_{G,{\mathcal{U}}}),$$
where for $\alpha=\om$ we write $s_\alpha:=\mathrm{id}_{\mathcal{M}(\mathcal{O}_{G,\mathcal{U}})}$ the identity in the multiplier algebra $\mathcal{M}(\mathcal{O}_{G,\mathcal{U}})$, is a tight representation.
\end{thm}

\begin{proof}
It is straight forward to check that $\pi$ is multiplicative, so we leave it to the reader. Note that in the case  $U^0 \in {\mathcal{U}}^0$, then $\pi(q_{(\om,U^0)})=p_{U^0}$  is the identity of ${ \mathcal{O}_{G,{\mathcal{U}}}}$ and $\{q_{(U^0,U^0)}\}$ is a cover for $\mathcal{E}({\mathcal{S}_{G,{\mathcal{U} }}})$. Otherwise, $U^0$ must be an  infinite set, say $U^0=\{v_1,v_2,\ldots\}$, and certainly $\{q_{(\{v_1\},\{v_1\})},q_{(\{v_2\},\{v_2\})},\ldots\}$ does not have a finite cover. Hence, in each case either condition (i) or (ii) of \cite[Proposition 11.7]{exe21} holds, and we may apply \cite[Proposition 11.8]{exe21} to show that $\pi$ is tight. So, assuming that ${\mathcal{C}} \subseteq \mathcal{E}({\mathcal{S}_{G,\mathcal{U} }})$ is a finite cover for an idempotent $q_{(\alpha,A)}$, we prove
\begin{equation}\label{(7.1)}
\pi(q_{(\alpha,A)})\leq  \bigvee_{q \in \mathcal{C}}{\pi(q)}.
\end{equation}
Since $\mathcal{C}\subseteq {\mathcal{J}}_{q_{(\alpha,A)}}$, Corollary \ref{cor4.4} says that each element of $\mathcal{C}$ is either of the form $q_{(\alpha,B)}$ with $B\subseteq A$ or of the form $q_{(\beta,B)}$ with $\beta= \alpha \gamma$ and $|\gamma|\geq 1$. Divide $ \mathcal{C}$ to the sets $X=\{q_{(\alpha,B)} \;:\; q_{(\alpha,B)} \in \mathcal{C}\}$ and $Y=\{q_{(\alpha\gamma ,B)}\in  \mathcal{C}\;:\; |\gamma|\geq 1\}$.
Note that $X$ is a finite cover for $q_{(\alpha,\bigcup_{q_{(\alpha,B)} \in {X}}{B})}$ and we have
$$\pi (q_{(\alpha,\bigcup_{q_{(\alpha,B) }\in {X}}{B}}) \leq \bigvee_{q_{(\alpha,B)} \in{ X}}{\pi(q_{(\alpha,B)})}.$$

Now we focus on idempotent $p:=q_{(\alpha,A\backslash  \bigcup_{q_{(\alpha,B)}\in{ X}}B)}$. Since $X$ is finite, $A\backslash  \bigcup_{q_{(\alpha,B)}\in{ X}}B$ belongs to $ \mathcal{U}^0$. Moreover, $A\backslash \bigcup_{q_{(\alpha,B)}\in{ X}}B$ must be a finite set of vertices because $Y$ covers idempotent $p$.  Indeed, if $A\backslash\bigcup_{q_{(\alpha,B)}\in{ X}}B=\{v_1,v_2,\ldots\}$ is infinite, then $q_{(\alpha,\{v_i\})}\leq p$ and the set $\{ q_{(\alpha,\{v_i\})} \}_{i=1}^{\infty}$ could not be covered by $Y$ (because $Y$ is finite and each element $q_{(\alpha \gamma,B)}$ in $Y$ meets exactly one $q_{(\alpha,\{v_i\})}$). So, write $A\backslash\bigcup_{q_{(\alpha,B)}\in{X}}B=\{v_1,\ldots,v_t\}$. Let $n=\max \{|\gamma|: q_{(\alpha\gamma,B)}\in Y\} $. Since $\mathcal{U}$ is row-finite and source-free, for any $1\leq i\leq t $, we may write by (CK4)

\begin{align}\label{(7.2)}
\pi(q_{(\alpha,\{v_i\})})&=s_{\alpha}p_{\{v_i\}}s_{\alpha}^*\\
&=\sum_{\substack{\gamma\in \mathcal{U}^n\\r(\gamma)=\{v_i\}\\  }}s_{\alpha\gamma}s_{\alpha\gamma}^*\nonumber\\
&=\bigvee_{\substack{\gamma\in \mathcal{U}^n\\r(\gamma)=\{v_i\}\\  }} \pi(q_{(\alpha\gamma,s(\gamma))}).\nonumber
\end{align}
Hence, each idempotent ${ q_{(\alpha,\{v_i\})}}$ is covered by $\{{ q_{(\alpha\gamma,s(\gamma))}}\;:\; \gamma \in {\mathcal{U}}^n \;\text{and}\; r(\gamma)=\{v_i\}\}$. Fix some $q_{(\alpha\gamma,s(\gamma))}$ where $r(\gamma)=\{v_i\}$ and $|\gamma|=n$. Since $q_{(\alpha\gamma,s(\gamma))}\leq q_{(\alpha,A)}$, there exists $q_{(\beta,B)}\in Y$ such that $q_{(\alpha\gamma,s(\gamma))} \Cap q_{(\beta,B)}$ (note that we have $|\beta|\leq |\alpha\gamma|$ by the definition of $n$). Then, $\beta$ is a prefix for $\alpha\gamma $ and two cases may occur: either $|\beta|< |\alpha\gamma|$ or $\beta=\alpha\gamma$ with $B\subseteq s(\gamma)$. If $\beta=\alpha\gamma$, we may consider the idempotent $q_{(\alpha\gamma,s(\gamma)\backslash B)} $ and repeat the above process for it. In the other case if $|\beta|< |\alpha\gamma|$, then $q_{(\alpha\gamma,s(\gamma))}\leq q_{(\beta,B)}$. Doing this argument inductively for all $\alpha\gamma$, where $\gamma \in {\mathcal{U}}^n$ and $r(\gamma)=\{v_i\}$, (\ref{(7.2)})  gives a finite number of idempotents $ Y_i \subseteq Y$ such that
 $$\pi(q_{(\alpha,\{v_i\}})\leq \bigvee_{q\in Y_i} \pi(q).$$
Therefore, we obtain
\begin{align}
\pi(q_{(\alpha,A)})&=\pi (q_{(\alpha,\bigcup_{q_{(\alpha,B) }\in {X}}{B})}+\sum_{i=1}^{t}\pi(q_{(\alpha,\{v_i\}}) \nonumber\\
&\leq( \bigvee_{q\in X}\pi(q)) \vee (\bigvee_{\substack{q\in Y_i\\1\leq i\leq t\\  }}\pi(q))\nonumber\\
&\leq( \bigvee_{q\in \mathcal{C}}\pi(q)), \nonumber
\end{align}
and (\ref{(7.1)}) is proved. Now \cite[Proposition 11.8]{exe21} (or alternatively \cite[Corolarly 2.3]{exe0045}) concludes that $\pi$ is a tight representation.
\end{proof}

Next goal is showing that the representation $\pi$ in Theorem \ref{thm7.2} is a universal tight one. First,  a lemma:

\begin{lem}\label{lem7.3}
Let $\psi :\mathcal{S}_{G,{\mathcal{U} }}\longrightarrow B(H)$ be a tight representation on a Hilbert space $H$  $($in the sense of {\cite[Defnition 13.1]{exe21}}$)$. If we define
$$S_e := \psi(e,s(e),{{1}_G},\om) ~ \text{and} ~  P_A:=\psi(\om,A,{{1}_G},\om)$$
for $e \in{\mathcal{U}}^1$ and $A\in {\mathcal{U}}^0$, then $\{S_e ,P_A\}$ is a Cuntz-Krieger $\mathcal{U}$-family in the sense of Definition \ref{defn2.3}.
\end{lem}

\begin{proof}
In order to verify  (CK1), we first claim that $P_{A\backslash B}=P_A-P_B$ for any $B\subseteq A \in  {\mathcal{U}}^0$. For convenience, let us write an idempotent $q_{(\om,A)}=(\om,A,1_G,\om)$ by $q_A$. Note that $\{q_B,q_{A\backslash B}\}$ is a cover for $q_A$. Indeed, if $q_{(\alpha,C)}\leq q_A$, we then have either
$$
\left\{
  \begin{array}{ll}
    r({\alpha})\subseteq A & \text{if}\;|\alpha|\geq 1,~ \text{or} \\
    C\subseteq A & \text{if}\; \alpha=\om
  \end{array}
\right.
~~ \Longrightarrow ~~
\left\{
  \begin{array}{lll}
    r({\alpha})\subseteq B &  \text{or} ~ r({\alpha})\subseteq {A\backslash B} & \text{if } |\alpha|\geq 1 \\
    C\cap B\neq \emptyset &  \text{or}  ~ C\cap  {A\backslash B}\neq \emptyset  & \text{if } \alpha=\om,
  \end{array}
\right.
$$
and consequently either $q_{(\alpha,C)}\Cap q_B$ or $q_{(\alpha,C)}\Cap q_{A\backslash B}$. Since idempotents $q_B ,\; q_{A\backslash B}$ are orthogonal, the tightness of $\psi$ implies
$$P_A=\psi(q_A) =\psi(q_B) \vee \psi(q_{A\backslash B})=\psi(q_B) + \psi(q_{A\backslash B})=P_B+P_{A\backslash B},$$
and the claim is proved.

Now we consider arbitrary $A,B \in {\mathcal{U}}^0$. By a similar argument as above, one may show that $\mathcal{C}=\{ q_{A \backslash A\cap B},\;q_{B \backslash A\cap B}, q_{A\cap B}\}$ is a cover for $q_{A\cup B}$ containing pairwise orthogonal idempotents. Therefore,
\begin{align}
P_{A\cup B}&=\psi(q_{A\cup B})\nonumber\\
&=\psi(q_{A\backslash A\cap B}) \vee \psi(q_{B\backslash A\cap B}) \vee \psi(q_{A\cap B}) \nonumber\\
&= P_{A\backslash A\cap B}+P_{B\backslash A\cap B}+P_{ A\cap B}\nonumber\\
&= (P_{A}-P_{ A\cap B})+(P_{B}-P_{ A\cap B})+P_{ A\cap B}\qquad\;\; (\text{by the above claim})\nonumber\\
&= P_{A}+P_{B}+P_{ A\cap B},\nonumber
\end{align}
 proving (CK1). Checking (CK2) and (CK3) are easy.

For (CK4), fix some singleton $\{v\} \in {\mathcal{U}}^0$ and let $r^{-1}(\{v\})=\{{e}_1,\ldots,{e}_n\}\subseteq {\mathcal{U}}^1$. Then $\{q_{({e}_i,s({e}_i))}\}_{i=1}^{n}$  is a cover for $q_{\{v\}}$. Indeed, given any $q_{(\gamma,A)}\leq q_{\{v\}}$, we have either $\gamma=\om$ and $A=\{v\}$, or $|\gamma|\geq 1 $ with $A\subseteq s(\gamma)$ by Corollary \ref{cor4.4}. In the former case, we have clearly $(q_{\{v\}}=)q_{(\gamma,A)}\Cap q_{({e}_i,s({e}_i))}$ for all $1\leq i \leq n$. In the other one, since $r(\gamma)=\{v\}$, there exists ${e}_i \in r^{-1}(\{v\})$ such that  $\gamma={e}_i \tau$ for some $\tau \in {\mathcal{U}}^*$, and thus $q_{(\gamma,A)}\Cap q_{({e}_i,s({e}_i))}$. Consequently, $\{q_{({e}_i,s({e}_i))}\}_{i=1}^{n}$ is a cover for $q_{\{v\}}$. As the idempotents $q_{({e}_i,s({e}_i))}$ are pairwise orthogonal, we can write
$$
P_{\{v\}}=\psi(q_{\{v\}})
=\bigvee_{i=1}^{n} \psi(q_{({e}_i,s({e}_i))}
=\sum_{i=1}^{n} \psi(q_{({e}_i,s({e}_i))})
= \sum_{i=1}^{n} S_{{e}_i}S_{{e}_i}^{*},$$
concluding (CK4) for $\{S_e,P_A\}$. We are done.
\end{proof}

Finally, we prove the main result of this section.

\begin{thm}\label{thm7.4}
Let $(G,\mathcal{U},\varphi)$ be a self-similar ultragraph as in Definition $\ref{defn3.2}$. Then the representation $\pi:\mathcal{S}_{G,{\mathcal{U} }}\longrightarrow \mathcal{O}_{G,{\mathcal{U} }}$ in Theorem \ref{thm7.2} is a universal tight representation, in the sense that if $\psi :\mathcal{S}_{G,{\mathcal{U} }}\longrightarrow \mathcal{A}$ is a tight representaition in a $C^*$-algebra $\mathcal{A}$, then there exists a unique $*$-homomorphism $\phi : \mathcal{O}_{G,{\mathcal{U} }}\longrightarrow\mathcal{A} $ such that $\psi=\phi \circ \pi$. Consequently, we have $\mathcal{O}_{G,{\mathcal{U} }} \cong C^{*}_{\mathrm{tight}}(\mathcal{S}_{G,{\mathcal{U} }})$.
\end{thm}

\begin{proof}
By Theorem \ref{thm7.2}, $\pi:\mathcal{S}_{G,{\mathcal{U} }}\longrightarrow\mathcal{O}_{G,{\mathcal{U} }}$ is a tight representation, so we prove the universality of $\pi$. Let $\psi :\mathcal{S}_{G,{\mathcal{U} }}\longrightarrow \mathcal{A}$ be a tight representation. For every $e \in \mathcal{U}^1$, $A \in {\mathcal{U} }^0$ and $g \in G$, define
$$S_e := \psi(e,s(e),{{1}_G},\om),\;\; U_{A,g} :=\psi(\om,A,g,\om) \text{ and } P_A:=U_{A,{{{1}_G}}}.$$
We will show that $\{S_e, U_{A,g}\}$ is a $(G,\mathcal{U})$-family in $\mathcal{A}$ by verifying the properties in Definition \ref{defn3.4}. First, Lemma \ref{lem7.3}  says $\{S_e, P_A: A\in \mathcal{U}^0, e\in \mathcal{U}^1\}$ is a Cuntz-Krieger $\mathcal{U}$-family. For relation (3) of Definition \ref{defn3.4}, one can write
$$(U_{A,g})^*=\psi(\om,A,g,\om)^*=\psi(\om,g^{-1}\cdot A,\;g^{-1},\om)=U_{{g^{-1}\cdot A},{g^{-1}}}.$$
Also, relation (4) follows from
\begin{align}
(U_{A,g})(U_{B,h})&=\psi(\om,A,g, \om)\psi(\om,B,h,\om)\nonumber\\
&=\psi\big((\om,A,g,\om)(\om,B,h,\om)\big)\nonumber\\
&=\psi(\om, A\cap (g\cdot B), gh,\om) \nonumber\\
&=U{_{{A\cap (g\cdot B)},{gh}}}.\nonumber
\end{align}

In order to prove (5), we fix $A\in {\mathcal{U} }^0$, $g \in G$ and $e \in {\mathcal{U} }^1$. Then
\begin{align*}
(U_{A,g})S_e & =\psi\big(\om,A,g, \om)\psi(e,s(e),{{1}_G}, \om\big)\\
&=\psi \big( (\om,A,g,\om)(e,s(e),{{1}_G}, \om)\big)\\
&=\left\{
    \begin{array}{ll}
      \psi (g\cdot e, g\cdot s(e), \varphi(g,e), \om) & \text{if } g\cdot r(e)\subseteq A \\
      0 & \text{otherwise}
    \end{array}
  \right.\\
&=\left\{
    \begin{array}{ll}
      S_{g\cdot e}U_{{g\cdot s(e)},{\varphi(g,e)}} & \text{if } g\cdot r(e)\subseteq A \\
      0 & \text{otherwise}.
    \end{array}
  \right.
\end{align*}
Therefore, $\{S_e,U_{A,g}\}$ is a $(G,\mathcal{U})$-family in $\mathcal{A}$ and the universality of $\mathcal{O}_{G,{\mathcal{U}} }$ gives a $*$-homomorphism
$\phi :  \mathcal{O}_{G,{\mathcal{U} }}\longrightarrow \mathcal{A}$ such that $\phi( s_{e})=S_{e}$ and $\phi({u}_{A,g})={U}_{A,g}$, satisfying $\psi=\phi \circ \pi $. Note that uniqueness of such homomorphism $\phi$ follows from that of universal $C^*$-algebra $\mathcal{O}_{G,{\mathcal{U} }}$.
\end{proof}

As said before, for each $s=(\alpha, A,g,\beta)$ in $\mathcal{S}_{G,{\mathcal{U} }}$ we have $s^*s=(\beta, g^{-1} \cdot A, 1_G, \beta)$ and $D^{s^*s}=\mathrm{Z}(\beta,g^{-1} \cdot A)$ (see Section \ref{sec80}). Moreover, Proposition \ref{prop8004} says that the compact open bisections
$$\Theta(\alpha, A,g,\beta)=[(\alpha, A,g,\beta),  \mathrm{Z}(\beta,g^{-1} \cdot A) ]$$
generate the topology on ${\mathcal{G} }_{\mathrm{tight}}(\mathcal{S}_{G,{\mathcal{U} }})$. Therefore, Theorem \ref{thm7.4} combining with \cite[Theorem 13.3]{exe21} implies the following.

\begin{cor}\label{prop7.5}
Let $(G,\mathcal{U}, \varphi)$ be a self-similar ultragraph. Then, there exists a $*$-isomorphism
$$\phi :  \mathcal{O}_{G,{\mathcal{U} }}\longrightarrow  C^{*}({\mathcal{G} }_{\mathrm{tight}}(\mathcal{S}_{G,{\mathcal{U} }}))$$
 such that
$$\phi(s_\alpha)={\mathbf{1}}_{\Theta(\alpha,s(\alpha),{{1}_G},\om) }
 \text{ and }
\phi(u_{A,g})={\mathbf{1}}_{\Theta(\om,A,g,\om)}$$
for all $\alpha \in{\mathcal{U}}^{\geq 1}$, $A \in {\mathcal{U} }^0$ and $g \in G$.
\end{cor}


\section{Hausdorffness of ${\mathcal{G} }_{\mathrm{tight}}(\mathcal{S}_{G,{\mathcal{U} }})$}

In this section, we investigate the Hausdorff property of the groupoid ${\mathcal{G} }_{\mathrm{tight}}(\mathcal{S}_{G,{\mathcal{U} }})$ via properties of the underlying self-similar ultragraph $(G,\mathcal{U},\varphi)$. Following \cite[Section 12]{bro14}, our key tool here is:
\begin{prop}[\cite{exe16}, Theorem 3.16]\label{prop8.1}
 If $\mathcal{S}$ is an inverse semigroup with zero, then the associated tight groupoid ${\mathcal{G} }_{\mathrm{tight}}(\mathcal{S})$ is Hausdorff if and only if every $s\in\mathcal{S}$ has a finite cover in $\mathcal{E}(\mathcal{S})$.
\end{prop}
Although one may try to extend the results of \cite[Section 12]{bro14} to the ultragraph setting, but we have two important differences which should be considered: first the graphs in \cite{exe16} are finite, while our ultragraphs here could contain infinitely many vertices and edges. Second, ultragraphs have set of vertices instead of single vertices in the directed graphs. These differences force us to change and extend the condition in \cite[Theorem 12.2]{bro14}.

\begin{defn}\label{defn8.2}
Let $(G,\mathcal{U},\varphi)$ be a self-similar ultragraph. A path $\alpha\in {\mathcal{U}}^{\geq1}$ is called \emph{strongly fixed by $g\in G$} if $g\cdot \alpha=\alpha$ and $\varphi(g,\alpha)=1 _G$.
\end{defn}

The next lemma shows that any extension of a strongly fixed by $g$ is again strongly fixed.

\begin{lem}\label{lem8.3}
Let $g\in G $. If $\alpha\in {\mathcal{U}}^{\geq1}$ is strongly fixed by g, then each its extension of the form $\gamma=\alpha\beta$, where $\beta\in {\mathcal{U}}^{\geq1}$,  is also strongly fixed by g.
\end{lem}

\begin{proof}
If $g\cdot \alpha=\alpha$ and $\varphi(g,\alpha)=1 _G$, we then have by definition
$$g\cdot\gamma=g\cdot (\alpha\beta) =(g\cdot \alpha)(\varphi(g,\alpha)\cdot \beta)=\alpha\beta=\gamma.$$
Moreover,
$$\varphi(g,\gamma)=\varphi(g,\alpha\beta)=\varphi(\varphi(g,\alpha),\beta)=\varphi(1_G,\beta)=1_G,$$
and hence $\gamma$ is strongly fixed by $g$.
\end{proof}

The preceding lemma says that if $\alpha \in {\mathcal{U} }^{\geq 1}$ is strongly fixed by $g$, then there exists a minimal strongly fixed subpath $\alpha_0$ of $\alpha$ such that $\alpha=\alpha_0{\alpha}'$. Given any $g \in G \backslash \{1_G\}$, minimal strongly fixed paths in ${\mathcal{U} }^{\geq 1}$ by $g$ have a key role for describing the Hausdorffness of ${\mathcal{G} }_{\mathrm{tight}}(\mathcal{S}_{G,{\mathcal{U} }})$.

\begin{lem}\label{lem8.4}
Let $(G,\mathcal{U},\varphi)$ be a self-similar ultragraph. For every $s=(\alpha,A,g,\beta)\in {\mathcal{S}_{G,\mathcal{U} }}$ and $(\gamma,B,1_G,\gamma) \in \mathcal{E}({\mathcal{S}_{G,\mathcal{U} }})$, we have $(\gamma,B,1_G,\gamma) \leq(\alpha,A,g,\beta)$ if and only if $\alpha=\beta$ and, moreover, one of the following conditions holds:

\begin{enumerate}
 \item $\gamma=\alpha\delta$ for some strongly fixed path $\delta\in  {\mathcal{U}}^{\geq 1}$ by $g$ with $r(\delta)\subseteq A$, or
  \item $\gamma=\alpha$, $B\subseteq A$ and $g=1_G $\;$ ($so $s$ is an idempotent$)$.
\end{enumerate}
\end{lem}

\begin{proof}
For the ``if " part, assume $\alpha=\beta$. If statement $(1)$ holds  and $\delta$ is strongly fixed by $g$, then
\begin{align*}
(\alpha,A,g,\beta)(\gamma,B,{{1}_G},\gamma) &=(\alpha(g \cdot \delta) ,\;(g \cdot s(\delta))\cap(\varphi(g,\delta)\cdot B),\;\varphi(g,\delta),\;\gamma) \\
&=(\alpha\delta, (g \cdot s(\delta))\cap B,\;{{1}_G},\;\gamma) \\
&=(\gamma, s(\gamma)\cap B,\;{{1}_G}, \gamma) \\
&=(\gamma,B,{{1}_G},\gamma),\nonumber
\end{align*}
and hence $(\gamma,B,{{1}_G},\gamma) \leq(\alpha,A,g,\beta)$. If $(2)$ holds, this inequality is trivial by Definition \ref{defn4.1}.

 For the ``only if " direction, assume $(\gamma,B,1_G,\gamma) \leq(\alpha,A,g,\beta)$, that is
 \begin{equation}\label{(8.1)}
(\alpha,A,g,\beta)(\gamma,B,1_G,\gamma)=(\gamma,B,1_G,\gamma).
\end{equation}
Since $(\alpha,A,g,\beta)(\gamma,B,1_G,\gamma)\neq 0$, one of the following cases may occur:

Case 1: $\gamma=\beta\delta$ with $|\delta|\geq 1$. Then
$$(\gamma,B,1_G,\gamma)=(\alpha,A,g,\beta)(\beta\delta,B,1_G,\gamma)=(\alpha(g\cdot \delta),(g\cdot s(\delta))\cap (\varphi(g,\delta)\cdot B), \varphi(g,\delta),\;\gamma),$$
so
$$\alpha(g\cdot \delta)=\gamma=\beta\delta \text{ and } \varphi(g,\delta)=1_G.$$
Since $|g\cdot \delta|=|\delta|$, it turns out $\alpha=\beta$, $g\cdot \delta=\delta$ and $\gamma=\alpha\delta.$

Case 2: $\beta=\gamma\delta$ with $|\delta|\geq 1$. This case is analogous to Case 1.

Case 3: $\beta=\gamma$. By Definition \ref{defn4.1}, in this case we have
$$(\alpha,A,g,\beta)(\gamma,B,1_G,\gamma)=(\alpha, A\cap (g \cdot B),g,\gamma)$$
which is equal to $(\gamma,B,1_G,\gamma)$ by (\ref{(8.1)}) above. So $\alpha=\gamma=\beta $, $g=1_G$ and $B\subseteq A$ as desired. This completes the proof.
\end{proof}

Note that for given any $s=(\alpha,A,g,\beta)\in {\mathcal{S}_{G,\mathcal{U} }}$ with $g\neq 1_G$, if $\gamma$ is a minimal strongly fixed by $g$ and $r(\gamma)\subseteq A$, then Lemma \ref{lem8.4} and Corollary \ref{cor4.4} imply that $q_{(\alpha\gamma,s(\gamma))}$ is a maximal idempotent in $\mathcal{J}_s=\{e\in \mathcal{E}(\mathcal{S}_{G,\mathcal{U}}) : e \leq s \}$.

We now prove the  main result of this section. For convenience, we denote by ${\mathbf{SF}}_g$ the set of strongly fixed paths by $g\in G$. Also, for each $A \in {\mathcal{U} }^0$ we use the notation
$$A{\mathbf{SF}}_g:=\{\gamma \in {\mathbf{SF}}_g : r(\gamma)\subseteq A\}.$$

\begin{thm}\label{thm8.5}
Let $(G,\mathcal{U},\varphi)$ be a self-similar ultragraph as in Definition $\ref{defn3.2}$. If for any $g\in G\backslash \{1_G\}$, each set of the form $\{v\}{\mathbf{SF}}_g$ or $s(e){\mathbf{SF}}_g$, where $v\in U^0$ and $e\in{\mathcal{U} }^1$, has finitely many minimal elements, then the tight groupoid ${\mathcal{G} }_{\mathrm{tight}}(\mathcal{S}_{G,{\mathcal{U} }})$ is Hausdorff.
\end{thm}

\begin{proof}
For every $g\in {G\backslash \{1_G\}}$, assume that every set of the form $\{v\}{\mathbf{SF}}_g$ or $s(e){\mathbf{SF}}_g$ contains finitely many minimal strongly fixed paths by $g$. By Proposition \ref{prop8.1}, it suffices to show that every $s \in \mathcal{S}_{G,{\mathcal{U} }}$ has a finite cover.

So, fix an arbitrary element $s=(\alpha,A,g,\beta)$ in $ \mathcal{S}_{G,{\mathcal{U} }}$. If $\alpha\neq\beta$, then $s$ dominates no idempotents by Lemma \ref{lem8.4}, and there is nothing to do. Moreover, whenever $g=1_G$ and $s=(\alpha,A,1_G,\alpha)$, $s$ itself is an idempotent, hence $\{s\}$ is a cover for $s$. Thus the nontrivial cases occur only for non-idempotent elements $s=(\alpha,A,g,\alpha)$ with $g\neq 1_G$. Arguments for the cases $\alpha =\om$ and $\alpha \in {\mathcal{U} }^{\geq 1}$ are similar, but have some partial differences. We will prove the statement for the case $\alpha \in {\mathcal{U} }^{\geq 1}$ and the proof in the other case is left to reader.

Thus assume $\alpha \in {\mathcal{U} }^{\geq 1}$. Lemma \ref{lem8.4}(1) says that the principal ideal $\mathcal{J}_s$ is equal to
$${\mathcal{J}_s}=\{q_{(\alpha\gamma,B)} : \gamma \in  A{\mathbf{SF}}_g  \; \text{and}\; B\subseteq s( \gamma)\}$$
where $q_{(\alpha\gamma,B)}=(\alpha\gamma,B,1_G,\alpha\gamma)$. We prove that
\begin{equation}\label{(8.2)}
\mathcal{C}=\{q_{(\alpha\tau,s(\tau))} : \tau \text{ is a minimal path in } A {\mathbf{SF}}_g \}
\end{equation}
is a finite cover for $s$. As noted before the theorem, $\mathcal{C}$ is in fact the set of all maximal idempotents dominated by $s$. Fix some $q_{(\alpha\gamma,B)}\leq s$ such that $\gamma \in A {\mathbf{SF}}_g$. By Lemma \ref{lem8.3}, there exists a minimal strongly fixed path $\tau \in {\mathbf{SF}}_g$ such that $\gamma = \tau {\gamma}'$. In particular, we have $q_{(\alpha\gamma,B)}\leq q_{(\alpha\tau,s(\tau))}$ by Corollary \ref{cor4.4}. Since $q_{(\alpha\tau,s(\tau))}\in \mathcal{C}$, we deduce that $\mathcal{C}$ is a cover for $s$.

It remains to show that $\mathcal{C}$ is a finite set. Since $A\in {\mathcal{U}}^0$, one may apply \cite[Lemma 2.12]{ana00} to find a finite set $X\subseteq U^0$ of vertices and a finite set of edges $Y\subseteq {\mathcal{U}}^1$ such that $A\subseteq X \cup (\bigcup_{e\in Y}s(e)).$
By hypothesis, $X \cup (\bigcup_{e\in Y}s(e))$ receives finitely many minimal strongly fixed paths by $g$, so dose $A$. Consequently, $\mathcal{C}$ is a finite set by (\ref{(8.2)}) and the proof is completed.
\end{proof}

 Note that the converse of Theorem \ref{thm8.5} is of course true. Indeed, for an element $s=(\om,s(e),g,\om)$ in $\mathcal{S}_{G,{\mathcal{U} }}$, Lemma \ref{lem8.4} yields
$${\mathcal{J}}_s=\{q_{(\alpha,A)} : \alpha\in s(e){\mathbf{SF}}_g,\; A\subseteq s(\alpha) \}.$$
Thus, $s$ has a finite cover if and only if $ s(e){\mathbf{SF}}_g$ has a finitely many minimal elements. A similar word can be also said for elements of the form $s=(\{v\},\{v\},g,g^{-1}\cdot \{v\})$, and therefore the desired statement is obtained by Proposition \ref{prop8.1}.


\section{Pseudo freeness and $E^*$-unitarity}

The notion of pseudo freeness is introduced in \cite{bro14} for self-similar graphs being compatible with $E^*$-unitary property of associated inverse semigroup. In this section, we investigate this property for inverse semigroups ${\mathcal{S}_{G,\mathcal{U} }}$ and extend \cite[Proposition 5.8]{bro14} and \cite[Theorem 3.3]{bur81} to the self-similar ultragraph setting.

\begin{defn}\label{defn5.1}
Let $(G,\mathcal{U},\varphi)$ be a self-similar ultragraph. We say a self-similar ultragraph $(G,\mathcal{U},\varphi)$ is {\it{pseudo free}} if for every $g \in G \backslash \{1_G\}$, there are no paths strongly fixed by $g$, that is
$$g\cdot \alpha\quad \text{and} \quad \varphi(g,\alpha)=1_G \quad  \Longrightarrow \quad g=1_G.$$
\end{defn}
Recall that an inverse semigroup $\mathcal{S}$ is said to be \emph{$E^*$-unitary} if any element $s$ in $\mathcal{S}$ which dominates a nonzero  idempotent is an idempotent itself; that is, if $s \in \mathcal{S}$ and $e \in \mathcal{E}({\mathcal{S}})$ satisfy $e\leq s$, then $s \in  \mathcal{E}({\mathcal{S}})$. So Lemma \ref {lem8.4} is a key tool here.

\begin{defn}[\cite{far05}]
Let $\mathcal{S}$ be an  inverse semigroup with zero. A \textit{prehomomorphism} of  $\mathcal{S}$ is a function
$$\theta : {\mathcal{S}} \backslash \{0\}\rightarrow H,$$
where $H$  is a group such that $\theta(st)=\theta(s)\theta(t)$ whenever $ st\neq 0$. A prehomomorphism $\theta$ is called {\it{idempotent pure}} if ${\theta}^{-1}(1_H)=\mathcal{E}(\mathcal{S})$. We know that every inverse semigroup $\mathcal{S}$ admits a prehomomorphism $\sigma :{\mathcal{S}} \backslash \{0\}\rightarrow U(\mathcal{S}) $ to the universal group $U(\mathcal{S})$ of $\mathcal{S}$. Then we say that $\mathcal{S}$ is \textit{strongly $E^{\ast}$-unitary} if $\sigma :{\mathcal{S}} \backslash \{0\}\rightarrow U(\mathcal{S}) $ is idempotent pure.
\end{defn}

Note that a strongly $E^*$-unitary inverse semigroup is certainly $E^*$-unitary, and one may combine \cite[Theorem 3.3]{bur81} and \cite[Proposition 5.8]{bro14} to imply that they are equivalent for the inverse semigroup of self-similar graphs.

Following \cite{bur81} to study strongly $E^*$-unitary of $\mathcal{S}_{G,\mathcal{U}}$, we construct a specific self-similar group $(G,{\tilde{\mathcal{U}}},{\tilde{\varphi}})$ from $(G,\mathcal{U},\varphi)$. To do this end, we collapse all vertices of $\mathcal{U}$ to a single vertex $v_0$, and construct the directed graph
$$\tilde{\mathcal{U}}=(\{{v_0}\},{\mathcal{U}}^1,\tilde{r},\tilde{s})$$
with one vertex such that $\tilde{r}(e)=\tilde{s}(e)= v_0$ for all $e\in {\mathcal{U}}^1$. Hence ${\tilde{\mathcal{U}}}^*$ is the monoid on ${\mathcal{U}}^1$ with the identity $v_0$. Moreover, the action $G \curvearrowright{\mathcal{U}}$ and 1-cocycle $\varphi$ induce naturally an action $G \curvearrowright\tilde{\mathcal{U}}$ and 1-cocycle ${\tilde{\varphi}}$  on ${G\times\tilde{\mathcal{U}}}^*$, making $(G,{\tilde{\mathcal{U}}})$ as a self-similar group in the sense of \cite{exe18}.

\begin{lem}[{\cite[Lemmas 3.1 and 3.2]{far05}}]\label{lem7..3}
Let $(G,\tilde{\mathcal{U}},\tilde{\varphi})$ be a self-similar group in the sense of \cite{exe18}. Then the associated  semigroup ${\mathcal{S}_{G,\tilde{\mathcal{U} }}}$ is strongly $E^*$-unitary if and only if $(G,\tilde{\mathcal{U}},\tilde{\varphi})$ is pseudo free.
\end{lem}

Now we prove the result.

\begin{thm}\label{thm10.4}
Let $(G,\mathcal{U},\varphi)$ be a self-similar ultragraph as in Definition $\ref{defn3.4}$. Then the following are equivalent:
\begin{enumerate}
  \item ${\mathcal{S}_{G,\mathcal{U} }}$ is $E^{\ast}$-unitary;
  \item ${\mathcal{S}_{G,\mathcal{U} }}$ is strongly $E^{\ast}$-unitary;
  \item $(G,\mathcal{U},\varphi)$ is pseudo free.
\end{enumerate}
\end{thm}

\begin{proof}
As $(2)\Rightarrow (1)$ holds for all inverse semigroup with zero, it suffices to prove $(1)\Rightarrow (3) \Rightarrow (2)$.

$(1)\Rightarrow (3)$: Let $(g,\alpha)\in G\times {\mathcal{U} }^{\geq 1}$ such that $\varphi(g,\alpha)={{1}_G}$ and $ g\cdot \alpha=\alpha$. If we define $s:=(\om,s(\alpha),g,\om)\in {\mathcal{S}_{G,\mathcal{U} }}$ and $e:=(\alpha,s(\alpha),{{1}_G},\alpha)\in \mathcal{E}({\mathcal{S}_{G,\mathcal{U} }})$, then
\begin{align*}
se&=(\om,s(\alpha),g,\om)(\alpha,s(\alpha),{{1}_G},\alpha) \\
&=(\alpha, (g \cdot s(\alpha))\cap  (\varphi(g,\alpha)\cdot    s(\alpha)), \varphi(g,\alpha), \alpha ) \\
&=(\alpha,g \cdot s(\alpha),1_G,\alpha)\qquad \text{(because $\varphi(g,\alpha)\cdot s(\alpha)\subseteq g\cdot s(\alpha)$ by Definition \ref{defn3.2}(3))} \\
&=(\alpha,s(g \cdot \alpha),1_G,\alpha) \\
&=(\alpha,s(\alpha),{{1}_G},\alpha) \\
&=e.
\end{align*}
So, by statement (1), $s$ must be an  idempotent in ${\mathcal{S}_{G,\mathcal{U} }}$ as well, which means $g={{1}_G}$. Therefore, $(G,\mathcal{U},\varphi)$ is pseudo free.

$(3)\Rightarrow (2)$: Consider the self-similar ultragraph $(G,\tilde{\mathcal{U}},\tilde{\varphi})$ constructed above. Since $(G,\mathcal{U},\varphi)$ is pseudo free, then so is $(G,\tilde{\mathcal{U}},\tilde{\varphi})$, and ${\mathcal{S}_{G,\tilde{\mathcal{U}} }}$ is strongly $E^{\ast}$-unitary by Lemma \ref{lem7..3}. Recall from \cite[Definition 4.1]{bro14} that the inverse semigroup ${\mathcal{S}_{G,\tilde{\mathcal{U}} }}$ is
$${\mathcal{S}_{G,\tilde{\mathcal{U}} }}=\{(\alpha,g,\beta) : g\in G, \; \alpha,\beta\in {\widetilde{\mathcal{U}}}^* \; \mathrm{and} \;  s(\alpha)=g \cdot s(\beta) \},$$
hence we may define $Q : \mathcal{U}^\sharp \rightarrow\tilde{\mathcal{U}}^*$ by
$$Q(\alpha)=
\left\{
  \begin{array}{ll}
    \alpha & \text{if } \alpha\in \mathcal{U}^{\geq 1} \\
    v_0 & \text{if}~ \alpha =\om
  \end{array}
\right.
$$
and the function $\tau  :{\mathcal{S}_{G,\mathcal{U} }}\rightarrow  {\mathcal{S}_{G,\tilde{\mathcal{U} }}}$ by
$$\tau(\alpha,A,g,\beta)=(Q(\alpha),g,Q(\beta)).$$

If $\sigma :{\mathcal{S}_{G,\tilde{\mathcal{U} }}}\rightarrow U({\mathcal{S}_{G,\tilde{\mathcal{U} }}})$ is the standard map from ${\mathcal{S}_{G,\tilde{\mathcal{U} }}}$ to its universal group, then the map
$$\sigma \circ \tau :{\mathcal{S}_{G,{\mathcal{U} }}}\rightarrow U({\mathcal{S}_{G,\tilde{\mathcal{U} }}})$$
is a prehomomorphism. Also, $\sigma$ is idempotent pure because ${\mathcal{S}_{G,\tilde{\mathcal{U} }}}$  is strongly $E^{\ast}$-unitary. Now fix $(\alpha,A,g,\beta)\in {\mathcal{S}_{G,{\mathcal{U} }}}$. If $\sigma \circ \tau(\alpha,A,g,\beta)={{1}_G}$, then this implies that $(Q(\alpha),g,Q(\beta))$ is a nonzero idempotent, thus $Q(\alpha)=Q(\beta)$ and $g={{1}_G}$. Hence $\alpha=\beta$ and $(\alpha,A,1_G,\beta)$ is an idempotent in $\mathcal{S}_{G,\mathcal{U}}$. This follows that $({\sigma  \circ \tau})^{-1}(\{{{1}_G}\})=\mathcal{E}({\mathcal{S}_{G,{\mathcal{U} }}})$, and $\sigma \circ \tau$ is an idempotent pure prehomomorphism. Consequently, ${\mathcal{S}_{G,{\mathcal{U} }}}$ is strongly $E^{\ast}$-unitary completing the proof.
\end{proof}

In the end, we remark that if an inverse semigroup $S$ is $E^*$-unitary, then $\mathcal{G}_{\mathrm{tight}}(S)$ is a Hausdorff groupoid \cite[Corollary 3.17]{exe16}. Hence, Theorem \ref{thm10.4} implies, in particular, that $\mathcal{G}_{\mathrm{tight}}(\mathcal{S}_{G,\mathcal{U}})$ is Hausdorff provided the self-similar ultragraph $(G,\mathcal{U},\varphi)$ is pseudo free. This may be also derived from Theorem \ref{thm8.5}. 


\subsection*{Acknowledgement}
This work was partially supported by Shahid Chamran University of Ahvaz under the grant number SCU.MM1403.279.


\end{document}